\documentclass{article}

\usepackage{graphicx}%
\usepackage{multirow}%
\usepackage{amsmath,amssymb,amsfonts}%
\usepackage{amsthm}%
\usepackage[title]{appendix}%
\usepackage{xcolor}%
\usepackage{textcomp}%
\usepackage{manyfoot}%
\usepackage{booktabs}%
\usepackage{algorithm}%
\usepackage{algorithmicx}%
\usepackage{algpseudocode}%
\usepackage{listings}%

\usepackage{amsfonts,amssymb}
\usepackage{url}
\usepackage{kbordermatrix}
\usepackage{caption,subcaption}
\usepackage{amsmath,amsfonts,graphicx,tikz,pgf,endnotes,multirow,blkarray,bigdelim,arydshln,scalefnt,verbatim}
\usepackage[dvipdfmx,dvisvgm]{animate}
\usetikzlibrary{decorations.text,shadings,shadows,shadows.blur,shapes.symbols}
\usetikzlibrary[patterns]

\usepackage[small,compact]{titlesec}
\usepackage[paper=a4paper,margin=1in]{geometry}
\usepackage{setspace}

\newtheorem{defin}{Definition}[section]

\newtheorem{lemma}[defin]{Lemma}

\newtheorem{theorem}[defin]{Theorem}

\newtheorem{corollary}[defin]{Corollary}

\newtheorem{example}[defin]{Example}

\newtheorem{remark}{Remark}

\title{Counting for rigidity under projective transformations in the plane}
\author{Leah Wrenn Berman \and Signe Lundqvist \and Bernd Schulze \and Brigitte Servatius \and Herman Servatius \and Klara Stokes \and Walter Whiteley}
\date{}

\begin{document}
\maketitle
\abstract{Let $P$ be a set of points and $L$ a set of lines in the (extended) Euclidean plane, and $I \subseteq P\times L$, where $i =(p,l) \in I$ means that point $p$ and line $l$ are incident. The incidences can be interpreted as quadratic constraints on the homogeneous coordinates of the points and lines. We study the space of incidence preserving motions of the given incidence structure by linearizing the system of quadratic equations. The Jacobian of the quadratic system, our projective rigidity matrix, leads to the notion of independence/dependence of incidences. Column dependencies correspond to infinitesimal  motions.  
Row dependencies or self-stresses allow for new interpretations of classical geometric incidence theorems. We show that self-stresses are characterized by a 3-fold balance. 
As expected, infinitesimal (first order) projective rigidity as well as second order projective rigidity imply  projective rigidity but not conversely. Several open problems and possible generalizations are indicated.}

\section{Introduction}
In the last half century there has been a careful exploration of what logical and algebraic forms the theorems of projective geometry take, building on  Cayley Algebra \cite{WhiteHand}.  The theorems can be expressed through products of invariant polynomials, which in turn must be expressible as meets (intersections)  and joins (unions) of brackets and algebraic combinations of these invariants \cite{WhiteHand, logic1, logic2, logic3, logic4}. This analysis also addressed the form of the reasoning for these theorems. That was, in part, the program of Gian-Carlo Rota -- looking for what properties would extend from finite subsets of plane projective geometry to the theory of dependence and independence of general matroids or pre-geometries.

The language of Cayley Algebra expresses the intersection of two projective subspaces $A$ and $B$ as $A\wedge B$ and the spanning subspace which is generated by $A$ and $B$ as $A\vee B$.  There is a fully developed algebra for combining these operations and using identities to determine when two expressions are equivalent \cite{WhiteHand}.  The underlying identities (\emph{syzygies} in the vocabulary of Sylvester)  built on determinants of matrices of shape $(n+1)\times (n+1)$ for dimension $n$ and identities for products of such determinants.  We will not be using this algebra in our reasoning here, though we will use the basic vocabulary in our figures of the configurations.

Here we will develop projective analogs of bar and joint rigidity, using the Jacobian of the incidence constraints of points and lines ({\em the projective rigidity matrix}), and the projective infinitesimal motions (the kernel of this Jacobian) and the projective self-stresses -- the row dependencies or co-kernel of this Jacobian.
As this presentation evolves, we will present evidence for the {\em conjecture}  that {\em all projective theorems are expressible as self-stresses of this projective rigidity matrix and conjunctions of such self-stresses}.


{\bf History  and Applications of Projective Geometry}
\label{applications}

Projective geometry has a rich history dating back to the 19th century. Sources of projective geometric questions and forms of expression can be found in David Henderson's book entitled: Experiencing Geometry on Plane and Sphere \cite{experience}. We mention here just a few applications:

\begin{itemize}
\item Geometric Rigidity Theory:  a comprehensive summary on the rigidity of 
 bar-joint frameworks and related structures through a projective lens is given in ~\cite{Schulze-Whiteley-2021}.

\item Engineering and Architecture: reciprocal diagrams and the Maxwell-Cremona correspondence; see  \cite{ran,maxwell1864xlv,WW82,CW93} for foundational work.

\item Approximation Theory: 
See~\cite{analogy}, for example.
 The survey paper~\cite{Schulze-Whiteley-2021} and its companion paper~\cite{Schulze-Whiteley-2024} summarize projective geometric connections of the developing work on geometric rigidity, in all dimensions, as well as other applications amenable to these approaches, such as splines in Approximation Theory. 

\item Computational Geometry: 
Problems in computer vision naturally involve collinearity and concurrence~\cite{KANATANI1991333}.

\item Klein's hierarchy: More transformations means fewer invariants.  Projective geometry has most transformations, and hence fewest invariants.

\end{itemize}

Each of these applications expresses the geometry and the analysis in a distinct way with its own questions.  Looking at what is common helps focus on the shared projective geometry and shared questions.

{\bf Preliminaries: }

\label{classical}
%
The real projective plane can be constructed
from $\mathbb{R}^3$
by letting the $1$-dimensional subspaces
of $\mathbb{R}^3$ define the points and the $2$-dimensional subspaces define the lines. This construction is called {\em projectivization} of the vector space. The homogeneous coordinates of a point in the projective plane are $(x_0:y_0:z_0)$ where $[x_0,y_0,z_0]$ is any vector in the corresponding $1$-dimensional subspace. The homogeneous coordinates of a line are $(a_0:b_0:c_0)$ where $a_0x+b_0y+c_0z=0$ is any (homogeneous) linear equation defining the corresponding $2$-dimensional subspace. A point with homogeneous coordinates $(x_0:y_0:z_0)$ is incident to (or {\em on}) a line  with homogeneous coordinates $(a_0:b_0:c_0)$ if $a_0x_0+b_0y_0+c_0z_0=0$. 

The projective plane has a duality; there is a one-to-one correspondence between the points and the lines that preserves the incidences. In terms of vector spaces, a projective duality is a map that sends a $1$-dimensional subspace of $\mathbb{R}^3$ to a $1$-dimensional subspace of the dual vector space. A vector $[a_0,b_0,c_0]$ of the dual vector space is a linear form, that is, a linear map from $\mathbb{R}^3$ to $\mathbb{R}$ sending $[x,y,z]\mapsto a_0x+b_0y+c_0z$. Each linear form $[a_0,b_0,c_0]$ defines a line of the projective plane, by requiring  $a_0x+b_0y+c_0z=0$.

{\em The affine plane:}
Let $l_{\infty}$ be any line of the projective plane  and call it the line at infinity. The set of all the points outside $l_{\infty}$ (the finite points) together with the set of all the lines except $l_{\infty}$ defines an affine plane. 
We select $l_{\infty}$ to be the line defined by the equation $z=0$. Then the homogeneous coordinates of a finite point $(x'_0:y'_0:z_0)$ can be normalized to $(x_0:y_0:1):=(\frac{x'_0}{z_0}:\frac{y'_0}{z_0}:1)$, where $(x_0,y_0)$ are the affine coordinates of the finite point. The projective dual of a finite point $(x_0:y_0:1)$ is a line with equation $x_0x+y_0y+1=0$. The affine plane is not preserved by the projective duality; all points have a dual line, but not all lines are duals of points. Specifically, lines through the origin of the affine plane are mapped to points of the line at infinity. By selecting another line to be $l_{\infty}$, one obtains another affine plane, but all these affine planes are isomorphic. 

{\em Incidence theorems:}
In projective geometry many results are concerned with the fact that some incidences are implied by other incidences. For example, the two following theorems are of fundamental importance. 
\begin{theorem}[Theorem of Desargues]
\label{desarguestheorem}
Two triangles are in perspective from a point if and only if they are perspective from a line. 
\end{theorem}

\begin{theorem}[Theorem of Pappus]
\label{pappustheorem}
The intersection points of opposite sides of a hexagon whose vertices alternate between two lines are collinear.
\end{theorem}

\section{Projective configurations}
A projective configuration of points and lines is a subset of points and a subset of lines of a projective plane with the induced incidence relation. 
The underlying combinatorial object of such a general configuration of points and lines is a bipartite graph $(P,L,I)$, with vertex set $P \cup L$ and edge set $I$.  Note that this bipartite graph has girth at least~6, since two distinct points in the projective plane uniquely determine a line  and two distinct lines intersect in a unique point.

In the classical literature, the name configuration of points and lines is often reserved for configurations  for which there are integers $k$ and $r$ such that there are $k$ points incident with every line and $r$ lines incident with every point \cite{Gru2009b}. A configuration with $v$ points and $b$ lines is called a $(v_r,b_k)$-configuration. When $r=k$, then also $v=b$ and the configuration is called a balanced configuration. A balanced configuration $v$ points and $k$ points on every line is called a $(v_k)$-configuration. 

In this paper, we study incidence preserving  motions of configurations of points and lines in the real projective plane. A point or a line in the projective plane has two degrees of freedom. An incidence between a point and a line constrains one degree of freedom, by requiring that a point remains on a line throughout the motion. In the following examples, we will consider the projective degrees of freedom of some familiar configurations of points and lines.

Loosely speaking, a set of incidences will be called \emph{dependent} if at least one of the incidences is implied by the other incidences, and hence does not constitute  any additional constraint on the system. Otherwise,  the incidences will be called \emph{independent}.

\subsection{The complete quadrangle and the complete quadrilateral}  
A complete quadrangle, see Figure~\ref{quadra}~b) is the $(4_3,6_2)$-configuration consisting of $4$ points in general position and the $6$ lines spanned by the $6$ pairs of points. There are $2$ points on every line and $6$ lines, hence $2\cdot 6=12$ incidences.

A complete quadrilateral, see Figure~\ref{quadri}~a) is the $(6_2,4_3)$- configuration consisting of $4$ lines in general position and their $6$ points of intersection. The complete quadrangle and the complete quadrilateral are projectively dual configurations; this implies that the number of incidences of the complete quadrilateral also is $12$. 

 The points in Figure~\ref{quadri} are labeled with their homogeneous coordinates. 
\begin{figure}[htb]
\centering
a) \includegraphics[width=.35\textwidth]{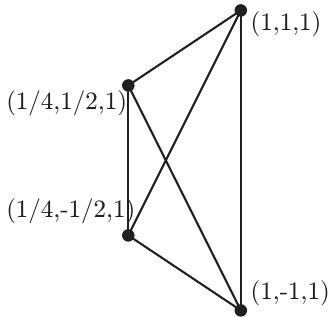}
\quad
b) \includegraphics[width=.35\textwidth]{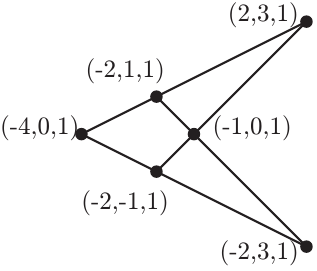}
\caption{a) Complete quadrangle  \label{quadri}
b) Complete quadrilateral\label{quadra}}
\end{figure}

The reader may easily verify that the coordinates of the four lines in Figure~\ref{quadri}~a) are the coordinates of the four points in 
Figure~\ref{quadra}~b) and the coordinates of the six lines in Figure~\ref{quadra}~b) are the coordinates of the six points in Figure~\ref{quadri}~a). Therefore these configurations are dual for the polarity defined by the quadratic form corresponding to the imaginary projective quadric with equation $x^2+y^2+z^2=0$. The matrix defining this quadratic form is the identity matrix.
We obtain a complete quadrilateral from a complete quadrangle by applying this polarity, thereby exchanging the role of points and lines. 

In Section~\ref{sec1} we will see how to set up the so-called projective rigidity matrix of a projective configuration using its coordinates. 
For the complete quadrilateral 
and the complete quadrangle this matrix has $ 2 \times 4 + 2\times 6 = 20$ columns and a row for each of the 12 incidences. In  Section~\ref{trivial} we will prove that the space of trivial projective motions has rank~8. Since $20-8= 12$ this implies that this small configuration is independent and has full rank -- it is projectively {\em isostatic}.



\subsection{The Pappus configuration}
The Pappus configuration (see Figure~\ref{pappus}) is the configuration of the $9$ points and the $9$ lines involved in Pappus' theorem (Theorem \ref{pappustheorem}).  

\begin{figure}[htb]
\centering
\captionsetup{width=0.9\textwidth}
\animategraphics[controls]{3.5}{figures/pappusmoveb}{10}{18}
\caption{A movable configuration. (The animation controls work depending on one's pdf viewer.) \label{pappus}}
\end{figure} There are three points on each line, and each point is the intersection of three of the lines. There are 27 incidences constraining  the $2 \times 9 + 2 \times 9 -8=28$
 degrees of freedom. If the incidences are independent, then the  Pappus configuration will have exactly one degree of freedom. We will show that there is more than one degree of freedom, indicating a dependence in the system.
 
 After pinning two points on two of the lines -- the white vertices in Figure~\ref{pappus} -- we are free to chose a third point on each one of the two lines and then uniquely complete the figure. So, there are two degrees of freedom in the construction, one for each choice of point. Therefore, the 27 incidences cannot be independent.


\subsection{The Desargues configuration}\label{desarguessubsec}

The Desargues configuration occurs as a result of the Desargues Theorem (Theorem \ref{desarguestheorem}), which states that two triangles are perspective from a point if and only if they are in perspective from a line.  The points of the configuration are the six points of the two triangles, $\{a,b,c\}$ and $\{a',b',c'\}$, the three intersection points of the corresponding sides, $\{d,e,f\}$, and the point of perspective.

The lines are the  six lines defining the sides of the two triangles, the three lines going through the point of perspective and one vertex from each triangle, and the line of perspectivity of the corresponding sides.  It is hence a configuration of $10$ points and  $10$ lines, with three points on every line and three lines incident to each point, a balanced $(10_3)$-configuration.

\begin{figure}[htb]
\centering
\captionsetup{width=0.9\textwidth}
a)\includegraphics[scale=.63]{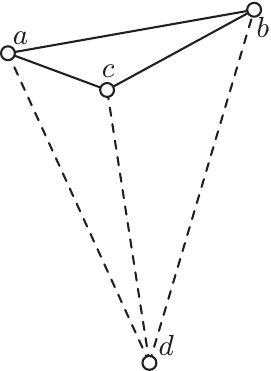}
b)\includegraphics[scale=.63]{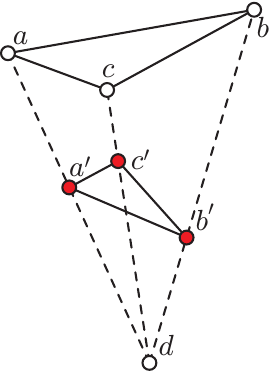}
c)\includegraphics[scale=.63]{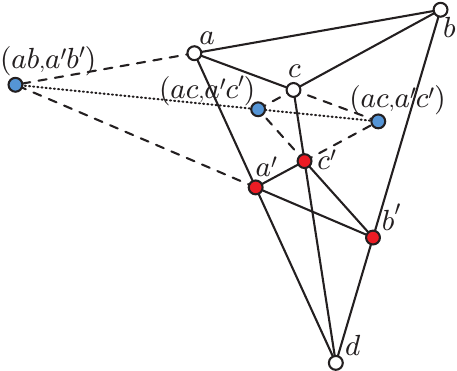}
\caption{The construction of a Desargues Configuration.  a) an initial isostatic $K_4$: $a,b,c,d$ pinned. (b) Adding three points (red) on three of the edges to add a second triangle $a',b', c'$ and c) constructing the line of perspective of the two triangles from three pairs of edges where the collinearity of the points is a constraint guaranteed by Desargues Theorem -- a constraint which is dependent on the other constraints.}
\end{figure}

From this construction, we see there were three free choices: $a',b',c'$ along fixed lines.  The free choices represent a three-dimensional space of non-trivial projective motions.
The configuration has $10$ points, $10$ lines and $30$, so there are $40-8=32$ degrees of freedom constrained by $30$ incidences. Because of Desargues's theorem, the incidences are dependent, so the space of motions is three-dimensional rather than two-dimensional. 

Desargues's theorem, just like Pappus's theorem, can be regarded as merely inducing constraints which we already have
on the list, hence asserting a dependence among the constraints, and increasing 
the degrees of freedom of the system.

\subsection{The Pascal configuration}\label{subsec:pascal}

Pascal's theorem states that if a hexagon is inscribed in a conic, then the three points in which opposite sides meet are collinear. See Figure \ref{mrpascallinefig}. 
Note that Pappus's theorem is a special form of Pascal's theorem as one can interpret two lines as a degenerate conic.

The line, whose existence is asserted by Pascal's theorem, is called the Pascal line of the hexagon. The configuration in Figure \ref{mrpascallinefig} 
\begin{figure}[htb]
    \centering    
    \captionsetup{width=0.9\textwidth}
a) \includegraphics[scale=.35]{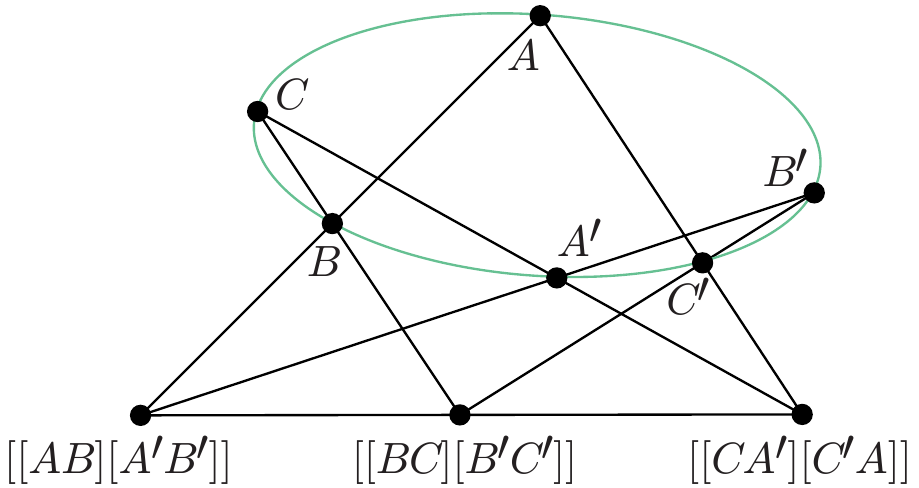}
b) \animategraphics[controls,scale=.35]{2.5}{figures/pascalanimb}{102}{111}
    \caption{a) A hexagon with a Pascal line. b) Animation with $A$, $B$, 
    $A'$, $B'$, pinned. $X$, $Y$, free to move subject to the collinearity $[[[AB][A'B']],X,Y]=0$. Points $C = [[BX][A'Y]]$, $C' = [[AY][B'X]]$ determined. \label{mrpascallinefig}}
\end{figure}
has $9$ points and $7$ lines, and $21$ incidences,
giving $2 \times 7 + 2 \times 9 - 8=24$ degrees of freedom, constrained by $21$ incidences. The space of non-trivial motions is expected to be three-dimensional. And, it is in fact three-dimensional. Starting out with four (fixed) points of the hexagon, the fifth point can be chosen arbitrarily, giving two degrees of freedom. The first five points determine a conic. There is one degree of freedom for the sixth point of the hexagon, as it needs to be placed on the conic determined by the first five points. The remaining three points, and the Pascal line, are determined by the initial six points on the conic, and the combinatorial structure - i.e., the choice of hexagon. So Pascal's theorem does not induce any extra constraints on the configuration in Figure \ref{mrpascallinefig}.

However, 6 points can be arranged in cyclic order in 60 different ways, so for a given set of 6 points on a conic there are 60 Pascal lines. Similarly, given the 6 conic points, there are 15 sets of 4 points to be partitioned into opposite sides of a hexagon in 3 different ways, giving 45 points. Since there are four possible ways to complete two opposite sides to a hexagon on the 6 points, each of the 45 points is the intersection point of 4 Pascal lines. We say that the 60 Pascal lines and the 45 points form the Pascal configuration. We count the number of incidences $I=3\times 60 = 4\times 45 = 180$, which is much smaller than $2(|P|+|L|)=2(60+45)$.

But this is not the end of the story. What is called Pascal's hexagrammum mysticum is a configuration of 95 points (60 Kirkman points, 20 Steiner points, 15 Salmon points) and 95 lines (60 Pascal lines, 20 Cayley lines, 15 Pl\"{u}cker lines) with $60\times 4+15\times 4 +20\times 7 =440$ incidences, which clearly must be dependent since $440 > 2(95+ 95)$. A good reference is
\cite{Demyst}.

 \section{Rigidity notions for the projective plane}

 \label{sec1}

\subsection{Realization space of a projective configuration}

Given an incidence geometry $S=(P,L,I)$, we  assign point coordinates $\mathbf{p_j}=(x_j:y_j:z_j)$ to each point $p_j \in P$, and line coordinates $\mathbf{l_i}=(a_i:b_i:c_i)$ to each line $l_i \in L$. The $\textit{realization space}$ of the incidence geometry is the real projective variety $V(S)$ defined by a quadratic equation

\begin{equation}
    a_ix_j+b_iy_j+c_iz_j=0
\end{equation}
\noindent
for each incidence $(l_i,p_j) \in I$. 

A point on the realization space is a realization of the incidence geometry as points and lines in the real projective plane. Note also that the realization space consists of \textit{all} configurations realizing the incidence geometry, including for example realizations where all points are given the same coordinates.

\subsection{Rigidity}

A configuration $(S, \mathbf{p}, \mathbf{l})$ is said to be \textit{rigid}, if all configurations in a neighborhood of $(S, \mathbf{p}, \mathbf{l})$ in the realization space can be obtained from $(S, \mathbf{p}, \mathbf{l})$ by a projective transformation. Otherwise, $(S, \mathbf{p}, \mathbf{l})$ is said to be \textit{flexible}. Equivalently, a configuration is flexible if there is a smooth path $(\mathbf{p}(t),\mathbf{l}(t))$ in $V(S)$ such that $\mathbf{p}(0)=\mathbf{p}$, $\mathbf{l}(0)=\mathbf{l}$, and there is some $t \in (0,1]$ for which $(S, \mathbf{p}(t), \mathbf{l}(t))$ cannot be obtained from $(S, \mathbf{p}, \mathbf{l})$ by a projective transformation \cite{proj_paper1}.  

\subsection{Infinitesimal rigidity}

Determining rigidity of configurations is in general difficult, involving
the solution of many quadratic equations. A more tractable approach is to consider the rank of the tangent space of the
realization space.

 Let  $\mathbf{l}_i$ be  finite line  given by its homogeneous coordinates
$\mathbf{l}_i=(a_i\colon b_i\colon 1)$, and  $\mathbf{p}_j$ be a finite point given by $\mathbf{p}_j=(x_j\colon y_j\colon 1)$. Then the  equation 
\begin{equation}
    a_i x_j + b_i y_j + 1=0
    \label{incidence_eq}
\end{equation}
describes that the point $\mathbf{p}_j$ and the line $\mathbf{l}_i$ are incident. Differentiating this constraint with respect to the time variable $t$, and evaluating at 0, then gives 
\begin{equation}
    \mathbf{p}_j\cdot \Delta \mathbf{l}_i + \mathbf{l}_i\cdot \Delta \mathbf{p}_j=0
    \label{Jacobian_eq}
\end{equation}
for each incidence $(p_j, l_i) \in I$
The coefficient matrix of this linear system is the \emph{(projective) rigidity matrix} $M(S,\mathbf{p}, \mathbf{l})$ whose row corresponding to the incidence $(p_j, l_i)$ has the following form:
\[
\left[\begin{array}{ c ccccc ccccc }
  &0&\dots  & x_j \ y_j &\dots&0&\ldots& a_i\  b_i &\ldots&0& 
\end{array}\right],
\]
where the $x_j,y_j$ entries are under the columns for $l_i$ and the $a_i,b_i$ entries are under the columns for $p_j$.

Irrespective whether or not an actual finite motion exists, a solution to the system \label{infmeaseq} is called
an {\em infinitesimal flex}.  If the only infinitesimal flexes correspond to the trivial
infinitesimal flexes arising from projective transformations, then the configuration
is said to be {\em infinitesimally rigid}. 

When defining the rigidity matrix, we assumed that the points of the configuration are finite and that the lines of the configuration do not go through the affine origin. A finite configuration of points and lines in the real projective plane is always projectively equivalent to a configuration with finite points and lines that do not go though the affine origin, so this assumption does not pose any restriction.

\section{Infinitesimal projective motions}

\label{trivial}

\subsection{Trivial infinitesimal motions}
The projective group of symmetries of the plane acts on the homogeneous points by the general linear group of $3 \times 3$ invertible matrices modulo the matrices $k\textup{Id}$, where $\textup{Id}$ is the $3 \times 3$ identity matrix. Given an invertible matrix $A$, the corresponding action
on the lines is by the inverse transpose of $A$.

The point group is generated by the nine types of $3\times 3$ elementary matrices, and we may find a basis for the space of
infinitesimal trivial motions by considering, for each type, a one-parameter family of them and take the derivative at the identity matrix.

\subsubsection{Dilations}
First, suppose we have one of the diagonal elementary matrices,
$$     \left[ \begin{array}{ccc}  1+t & 0 & 0 \\ 0 & 1 & 0 \\ 0 & 0 & 1  \end{array} \right],
\left[ \begin{array}{ccc}  1 & 0 & 0 \\ 0 & 1+t & 0 \\ 0 & 0 & 1  \end{array}\right],
\left[ \begin{array}{ccc}  1 & 0 & 0 \\ 0 & 1 & 0 \\ 0 & 0 & 1+t  \end{array} \right], $$
which correspond to a dilation by a factor $1+t$. When $t = 0$, the transformation is the identity.
The corresponding infinitesimal action is $$ \left.\frac{d}{dt}\right|_{t=0} \left[ \begin{array}{ccc}  1+t & 0 & 0 \\ 0 & 1 & 0 \\ 0 & 0 & 1  \end{array} \right]
= \left[ \begin{array}{ccc}  1 & 0 & 0 \\ 0 & 0 & 0 \\ 0 & 0 & 0  \end{array} \right],
     \left.\frac{d}{dt}\right|_{t=0} \left[ \begin{array}{ccc}  1/(1+t) & 0 & 0 \\ 0 & 1 & 0 \\ 0 & 0 & 1  \end{array} \right]
   =  \left[ \begin{array}{ccc}  -1 & 0 & 0 \\ 0 & 0 & 0 \\ 0 & 0 & 0  \end{array} \right] $$
giving an infinitesimal dilation in the $x$ direction for the points, and a complementary one for the lines.  See Figure~\ref{dilateshearfig}a.
For the second matrix we get the $y$-dilations:
$$
\left[x_1,0,x_2,0,\ldots ; -a_1,0, -a_2,0,  \ldots  \right],
 \ 
 \left[0,y_1,0,y_2,\ldots ; 0, -b_1,0, -b_2, \ldots  \right]
 $$
\begin{figure}[htb]
\centering
a)\includegraphics[width=.43\textwidth]{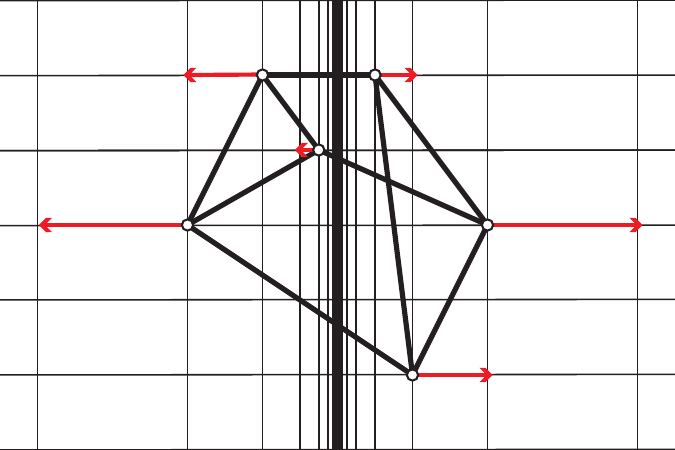}
\
b)\includegraphics[width=.43\textwidth]{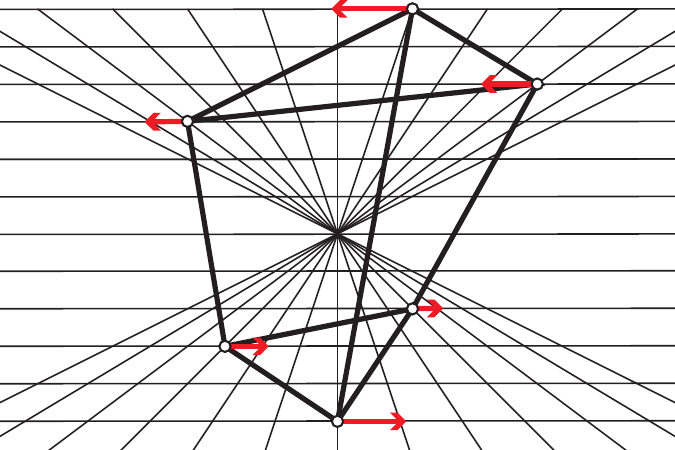}

\caption{a) An infinitesimal $x$-dilation, b) An infinitesimal $x$-shear\label{dilateshearfig}}
\end{figure}

\begin{remark}
The reciprocal point $(-a/(a^2 + b^2): -b/(a^2 + b^2):1)$ of the line $(a:b:1)$
is dilated to $ (-at/(a^2 + b^2): -b/(a^2 + b^2):1)$, which is on the dilated line, as required, but it is not its reciprocal point. The
reciprocal point of the dilated line
is $(\frac{-at}{a^2 + b^2t^2}: \frac{-bt^2}{a^2 + b^2t^2}:1)$.
\end{remark}

The third matrix is more complicated since it does not preserve $z = 1$. 
However, since the sum of the three diagonal elementary matrices is a pure dilation at the origin, which is the identity in projective space. So the other
$8$ elementary matrices generate the $8$ dimensional space
of trivial infinitesimal projective motions. Therefore, we need not consider the infinitesimal action of the third matrix here.


\subsubsection{Shears}
The elementary matrix 
$$\left[ \begin{array}{ccc}  1 & 0 & 0 \\ t & 1 & 0 \\ 0 & 0 & 1  \end{array} \right]$$
and its inverse transpose
$$\left[ \begin{array}{ccc}  1 & -t & 0 \\ 0 & 1 & 0 \\ 0 & 0 & 1  \end{array}\right] $$
act on the points and lines respectively by shearing the point in the $y$ direction, and shearing the line in the negative $x$ direction.

The corresponding infinitesimal action
$$ \left.\frac{d}{dt}\right|_{t=0} \left[ \begin{array}{ccc}  1 & 0 & 0 \\ t & 1 & 0 \\ 0 & 0 & 1  \end{array} \right]
=\left[ \begin{array}{ccc}  0 & 0 & 0 \\ 1 & 0 & 0 \\ 0 & 0 & 0  \end{array} \right],
 \quad
   \left.\frac{d}{dt}\right|_{t=0}  \left[ \begin{array}{rrr}  1 & -t & 0 \\ 0 & 1 & 0 \\ 0 & 0 & 1  \end{array} \right] =
 \left[ \begin{array}{rrr}  0 & -1 & 0 \\ 0 & 0 & 0 \\ 0 & 0 & 0  \end{array} \right] $$
is infinitesimal action in which the point is sheared in the $y$ direction, while the line is sheared in the
negative $x$ direction, giving $\left[0,x_1,0,x_2,\ldots ; -b_1,0, -b_2,0  \ldots  \right]$ for our configuration.

The infinitesimal action of the elementary matrix 
$$\left[ \begin{array}{ccc}  1 & t & 0 \\ 0 & 1 & 0 \\ 0 & 0 & 1  \end{array} \right]$$
and its inverse transpose
$$\left[ \begin{array}{ccc}  1 & 0 & 0 \\ -t & 1 & 0 \\ 0 & 0 & 1  \end{array}\right] $$
gives the infinitesimal $x$-shears of the points,
 $$\left[y_1,0,y_2,0,\ldots ; 0,-a_1,0, -a_2,  \ldots  \right].$$
See Figure~\ref{dilateshearfig}b.

\subsubsection{Translations and rotations at infinity}
Next we consider the elementary matrix 
$$\left[ \begin{array}{ccc}  1 & 0 & t \\ 0 & 1 & 0 \\ 0 & 0 & 1  \end{array} \right]$$ and its inverse transpose 
$$\left[ \begin{array}{ccc}  1 & 0 & 0 \\ 0 & 1 & 0 \\ -t & 0 & 1  \end{array} \right]$$
which act as a translation in the direction of the $x$-axis.
The infinitesimal action becomes
$$ \left.\frac{d}{dt}\right|_{t=0} \left[ \begin{array}{ccc}  1 & 0 & t \\ 0 & 1 & 0 \\ 0 & 0 & 1  \end{array} \right]
= \left[ \begin{array}{ccc}  0 & 0 & 1 \\ 0 & 0 & 0 \\ 0 & 0 & 0  \end{array} \right],
     \left.\frac{d}{dt}\right|_{t=0} \left[ \begin{array}{ccc}  1 & 0 & 0 \\ 0 & 1 & 0 \\ -t & 0 & 1  \end{array} \right]
   =  \left[ \begin{array}{ccc}  0 & 0 & 0 \\ 0 & 0 & 0 \\ -1 & 0 & 0  \end{array} \right]$$
The first matrix translates each point in the $x$ direction and preserves the plane $z=1$, but the second matrix does not. The second matrix sends the line $(a:b:1)$ to the line $(a:b:1-at)$.  The projection to $z=1$ is $(a/(1-at): b/(1-at): 1)$, and the derivative at $t=0$ is $(-a^2, -ab)$.
This gives the infinitesimal point translations in the $x$-direction. Translation in the $y$-direction corresponds to the elementary matrix
with $A_{23} = t$.
$$
\left[1,0,1,0,\ldots ; a_1^2,a_1b_1, a_2^2,a_2b_2, \ldots  \right],
\ 
\left[0,1,0,1,\ldots ; a_1b_1, b_1^2,a_2b_2, b_2^2, \ldots  \right].
$$

Next we consider

the two matrices which give the action on the points by
$$ \left[ \begin{array}{ccc}  1 & 0 & 0 \\ 0 & 1 & 0 \\ t & 0 & 1  \end{array} \right],
     \left[ \begin{array}{ccc}  1 & 0 & 0 \\ 0 & 1 & 0 \\ 0 & t & 1  \end{array} \right] $$
which simply give the dual of the case just considered.
The previous transformations on the points all preserved the line at infinity,
however these last two move beyond that.
Here the action on the line coordinates is an infinitesimal translation in the
$x$ and $y$ directions, respectively, 
so we know immediately from the previous calculation what the infinitesimal
translations must be:
$$
\left[-x_1^2, -x_1y_1, -x_2^2, -x_2y_2, \ldots ; 1, 0, 1, 0,   \ldots  \right],
\left[-x_1y_1, -y_1^2, -x_2y_2, -y_2^2,  \ldots ; 0, 1, 0, 1,    \ldots  \right],
$$
It acts on the points by an infinitesimal rotation at infinity, see Figure~\ref{infrotfig}.
\begin{figure}[htb]
\centering
\includegraphics[width=.75\textwidth]{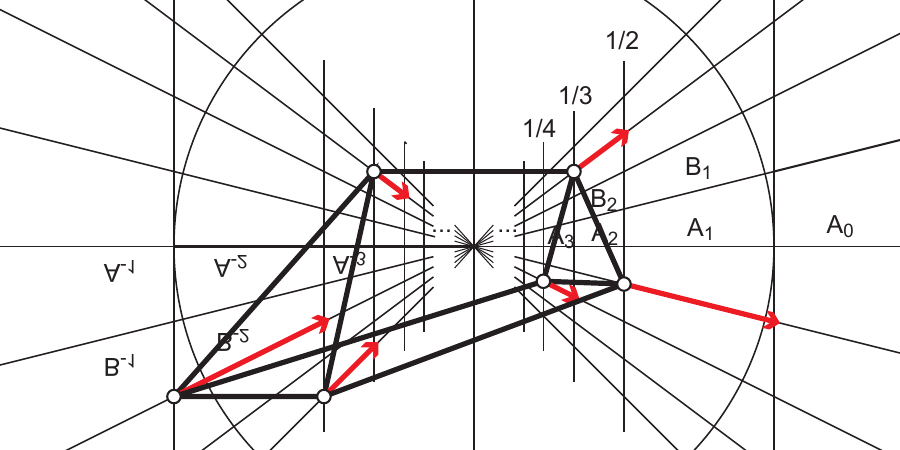}
\caption{An infinitesimal rotation of the picture plane about an axis in the picture plane.\label{infrotfig}}
\end{figure}
\subsection{Geometric interpretation of an infinitesimal motion.}

An infinitesimal motion of a projective configuration $(S,\mathbf{p},\mathbf{l})$, with  $S=(P,L,I)$,  is an assignment of a vector $\Delta  \mathbf{p}$ to each point $p\in P$ and a vector $\Delta \mathbf{l}$ to each line $l\in L$, such that
$\mathbf{p} \cdot \Delta \mathbf{l} + \mathbf{l} \cdot \Delta  \mathbf{p} = 0$ whenever $p$ and $l$ are incident. 

Suppose that an incident point-line pair $(p,l)$ and an infinitesimal motion $\Delta \mathbf{l}$ of the realization $\mathbf{l}$ of the line $l$ are given. In Figure \ref{infinterp03fig}, the vector $\Delta \mathbf{l}$ is represented as a displacement of the vector of homogeneous coordinates $\mathbf{l}$ of the horizontal line $l$.  What are the allowable infinitesimal
displacements $\Delta \mathbf{p}$ of the realization $\mathbf{p}$ of $p$?

The component of $\Delta \mathbf{p}$ in the direction of the line $l$ is perpendicular to the vector $\mathbf{l}$, hence unconstrained by the incidence condition $\mathbf{p} \cdot \Delta \mathbf{l} + \mathbf{l} \cdot \Delta  \mathbf{p} = 0$. Consider now the effect
on the perpendicular component, $\Delta \mathbf{p}_\perp$.
First assume that $\Delta \mathbf{l}$ is not parallel to $\mathbf{l}$.
 Let $p_0$ be the point of intersection of the line $l$ and the line with normal $\Delta\mathbf{l}$ going through the origin of the affine plane.

Then $\mathbf{p}_0 \cdot \Delta \mathbf{l}=0$, so
the incidence condition becomes $\mathbf{l}  \cdot \Delta\mathbf{p}_0 = 0$ implying that $\Delta \mathbf{p}_{0\perp} = \mathbf{0}$.  Therefore $\Delta\mathbf{p}_{0\perp}$ is the infinitesimal motion of a rotation fixing the point $\mathbf{p}_0$.
  See Figure~\ref{infinterp03fig}.

\begin{figure}[htb]
\centering
\captionsetup{width=0.9\textwidth}
\includegraphics[width=.9\textwidth]{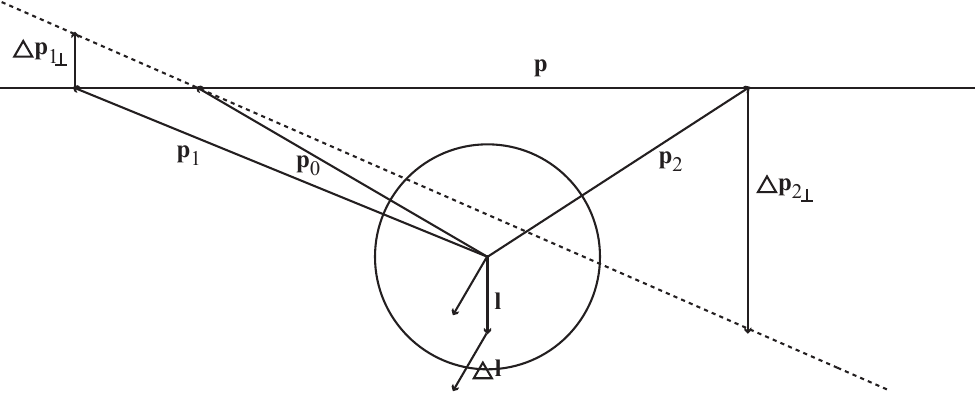}
\caption{Constructing infinitesimal displacements of points  $p_1,p_2$ on a moving line with the normal $\mathbf{l}$. \label{infinterp03fig}
}
\end{figure}

Setting $\mathbf{p} = \mathbf{p}_0 + (\mathbf{p} - \mathbf{p}_0)$, the constraint gives
$(\mathbf{p} - \mathbf{p}_0) \cdot \Delta \mathbf{l} + \mathbf{l} \cdot \Delta  \mathbf{p} = 0$, so
$$ | \Delta \mathbf{p}_{\perp}| =  |(\mathbf{p} - \mathbf{p}_0)||\Delta \mathbf{l}|/|\mathbf{l}|$$
and the magnitude of the perpendicular infinitesimal displacement of $\mathbf{p}$ is proportional to
the distance from the motion center, $\mathbf{p}_0$.
(It may be convenient in examples to compute the perpendicular displacement $\Delta \mathbf{p}_\perp$ at the
point of $l$ closest to the origin, $\Delta \mathbf{p}_\perp =  -\mathbf{l}\cdot \Delta \mathbf{l} /\mathbf{l}^2$
as a reference to construct the rest.) 

The infinitesimal motion of a point $\mathbf{p}$ on $\mathbf{l}$ is the sum of the component of $\Delta \mathbf{p}$ in the direction of $l$ with $\Delta \mathbf{p}_\perp $, the component perpendicular to $l$.
If $\Delta\mathbf{l}$ is parallel to $\mathbf{l}$ then $\Delta \mathbf{p}_{\perp} = \mathbf{0}$ for all points $p$ on the line, so that $\Delta\mathbf{p}$ is equal to $\Delta\mathbf{l}$.

\begin{example}\rm   The Desargues configuration  has a $3$-dimensional space of non-trivial infinitesimal motions. The generators illustrated in Figure~\ref{infintesimalexample07fig} 
   \begin{figure}[htb]
   \centering 

   a) \includegraphics[width=.32\textwidth]{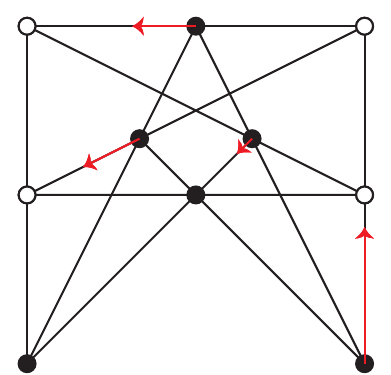}
   \quad 
   b)\includegraphics[width=.32\textwidth]{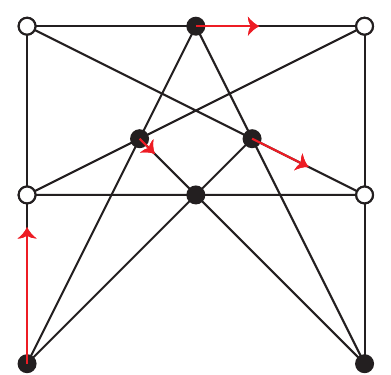}
   \caption{Generators of the space of infinitesimal motions of the Desargues configuration \label{infintesimalexample07fig}}
   \end{figure}
   and~\ref{infintesimalexample07figb}a
   can be found by first pinning the white vertices. With that choice of pins, the three lower black vertices each have only a
   one-dimensional space of movement possible.  Pinning two of the three of them and choosing an infinitesimal displacement for the third, allows one to solve uniquely for the other three non-zero displacements, as the reader may easily check.  It is also easy to see that fixing infinitesimally all three lower black vertices forces the remaining three to be fixed as well.  
   (In the figures, only the point displacements are indicated, since the line displacements can be inferred from these.) 

   Notice that these displacements are infinitesimal only, although each does roughly correspond to a finite motion,
   See Figure~\ref{infintesimalexample07figb}b. 
      \begin{figure}[htb]
   \centering 

   a) \includegraphics[width=.32\textwidth]{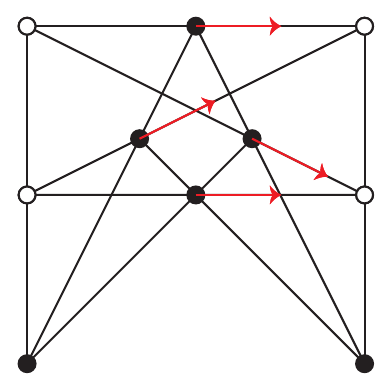}
   \quad 
   b)\includegraphics[width=.32\textwidth]{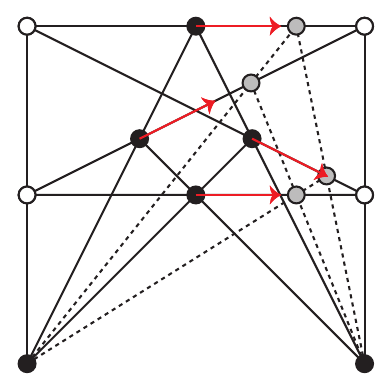}
   \caption{An infinitesimal motion versus a finite motion. \label{infintesimalexample07figb}}
   \end{figure}
   There only four of the non-pinned  infinitesimal displacements are realized in the finite motion, the other two being
   well short of the mark.
\end{example}

\subsection{Higher dimensional trivial motions}

So far we have looked at the $8$-dimensional space of trivial motions in dimension $d=2$ (the plane) -- the kernel of the infinitesimal rigidity matrix.

We can generalize these counts to all dimensions.

$d=1$ (the line):  The projective transformations are $2\times 2$ invertible matrices modulo the multiples
if the $2\times 2$ identity matrix.This gives dimension $4-1= 3$ for the trivial motions.  You can choose the location of up to
three points along the line provided they do not coincide.

$d=3$: The projective transformations are represented by $4\times 4$ invertible matrices, modulo multiples
of the $4\times 4 $ identity matrix, giving a dimension of $16-1 = 15 = 5\times 3$.  Every point is represented by $4$
homogeneous coordinates, or $3$ coordinates with the last coordinate 1.  The factoring $15=5\times 3$ tell
us that we can transform any 5 general position points (no 4 points coplanar).
The complete hypergraph of 5 pts and all possible triangles has the counts $|P|=5, |L|=10, |I|= 5 \times 6=30 = 3 \times 15 -15$.  The count suggests this is a maximal independent set of rows, or isostatic.

In  a general dimension $d$, following the same pattern as above, the projective transformations are represented by
 $(d+1)\times (d+1)$ matrices modulo multiples of the identity: $(d+1)^2-1=d^2-2d =  d(d+2)$.  Given that,
 modulo the multiples for homogeneous coordinates,  we have $d$ coordinates for each point,  $ d(d+2) $
 tells us that we can choose $d+2$ general position points to go onto $d+2 $ general position other points by a projective transformation.  Equivalently, we can choose $d+2$ hyperplanes to take to $d+2$ hyperplanes.

 The overall counts for a complete hypergraph in dimension $d$ on $d+2$ vertices is  $|P|=d+2, |L|= \binom{d+2}{2} , |I| = (d+2)\binom{d+1}{2}$.


\section{Pinning infinitesimal projective motions}

Suppose we have an incidence preserving motion of a projective configuration with at
least $4$ points, no three of which are collinear in a neighborhood of $0$: say $\mathbf{p}_0(t)$, $\mathbf{p}_1(t)$, $\mathbf{p}_2(t)$, $\mathbf{p}_3(t)$.
For all $t$ in a neighborhood of $0$, the point $\mathbf{p}_{0}(t)$ lies in the span of
$\{\mathbf{p}_{1}(t),\mathbf{p}_{2}(t),\mathbf{p}_{3}(t)\}$,
so let $\lambda_{i,t}$ be chosen so that $\mathbf{p}_0(t) = \lambda_{1,t} \mathbf{p}_1(t) + \lambda_{2,t} \mathbf{p}_2(t) + \lambda_{3,t} \mathbf{p}_3(t)$.
Define the matrix $A_t$ as the unique matrix sending $\lambda_{i,t} \mathbf{p}_{i,t}$ to
$\lambda_{i,0} \mathbf{p}_{i,0}$, for $i \in \{1,2,3\}$.  For $i \in \{1,2,3\}$, the matrix $A_t$ maps $\mathbf{p}_{i,t}$
to $\mathbf{p}_{i,0}$ up to scalar multiple, and maps
  $\mathbf{p}_{0,t}$
to $\mathbf{p}_{0,0}$ precisely. Now the composite motion defined by $A_t\mathbf{p}_{i,t}$ on all points, and by $(A_t^{-1})^T\mathbf{k}_{j,t}$
on the lines fixes the four given points.

\begin{theorem}
   Every motion of a projective configuration containing four points, no three of which are collinear,
   is equivalent to one in which those four points are pinned.

   Equivalently, we can pin at four lines, no three of which are coincident.

   We also can pin the three non-collinear points and a line not passing through any of them, or
   at three non-coincident lines and a point not on any of them.
\end{theorem}

\begin{proof}
   The second pinning criterion is dual to the first, proved by the preceding argument.
   
   If you pin at three non-collinear points,
   and a line avoiding all of them, then the lines of the triangle through the pinned points, $p_1$, $p_2$, and $p_3$, meets that line at $p_{12}$, $p_{23}$ and $p_{31}$, forming a pinned
   complete quadrilateral, and that quadrilateral is equivalently pinned by four of the points,
   say $p_1$, $p_2$, $p_{31}$ and $p_{23}$.
\end{proof}
We say that a line is fixed by a projective transformation if the set of points on the line is fixed setwise, not necessarily pointwise. 
\begin{theorem}\label{pinningagain}
    Requiring that any two points and any two lines in projective space be fixed still leaves at least one projectivity preserving all four elements.
\end{theorem}

\begin{proof}

   Let $p$ and $q$ be the points, and $l$ and $m$, be the lines.  Without loss of generality, assume they are
   all distinct, and that the points do not lie on the lines.
   Let $n_1 = p \wedge q$ be the
   line through $p$ and $q$, and $n_2$ be any fourth distinct line.  Form the complete quadrangle on these four lines, and let $r_1$ and $r_2$ be the points of
   intersection of $l$ and $m$ and $n_1$ and $n_2$ respectively.  Without loss of generality, assume that the line joining $r_1$ and $r_2$ is
   the line at infinity, that implies that $l$ and $m$ are parallel, and that $n_1$ and $n_2$ are also parallel.

   Again, using a shear, we can assume without loss of generality that the parallelogram is, in fact, a rectangle.  Now any dilation fixing the line
   $n_1 = p \wedge q$ and dilating in the perpendicular direction will fix $p$, $q$, $l$ and $m$.  So fixing those four elements leaves at least one
   projectivity.
\end{proof}

So two points and two lines never comprise a complete projective pinning system.

If the only pinned motion is the constant motion, then the configuration is rigid.
For example, the configuration of Figure~\ref{mrfourteenfig}
has many motions.  The white vertices can each independently move in a neighborhood and, given those, the lines and black
vertices may be computed.  To see the non-trivial degree of freedom, we may pin the vertices $a$, $b$, $c$, and $d$,
and conclude that the true degree of freedom is six.
\begin{figure}[htb]
\centering
a) \includegraphics[width=.4\textwidth]{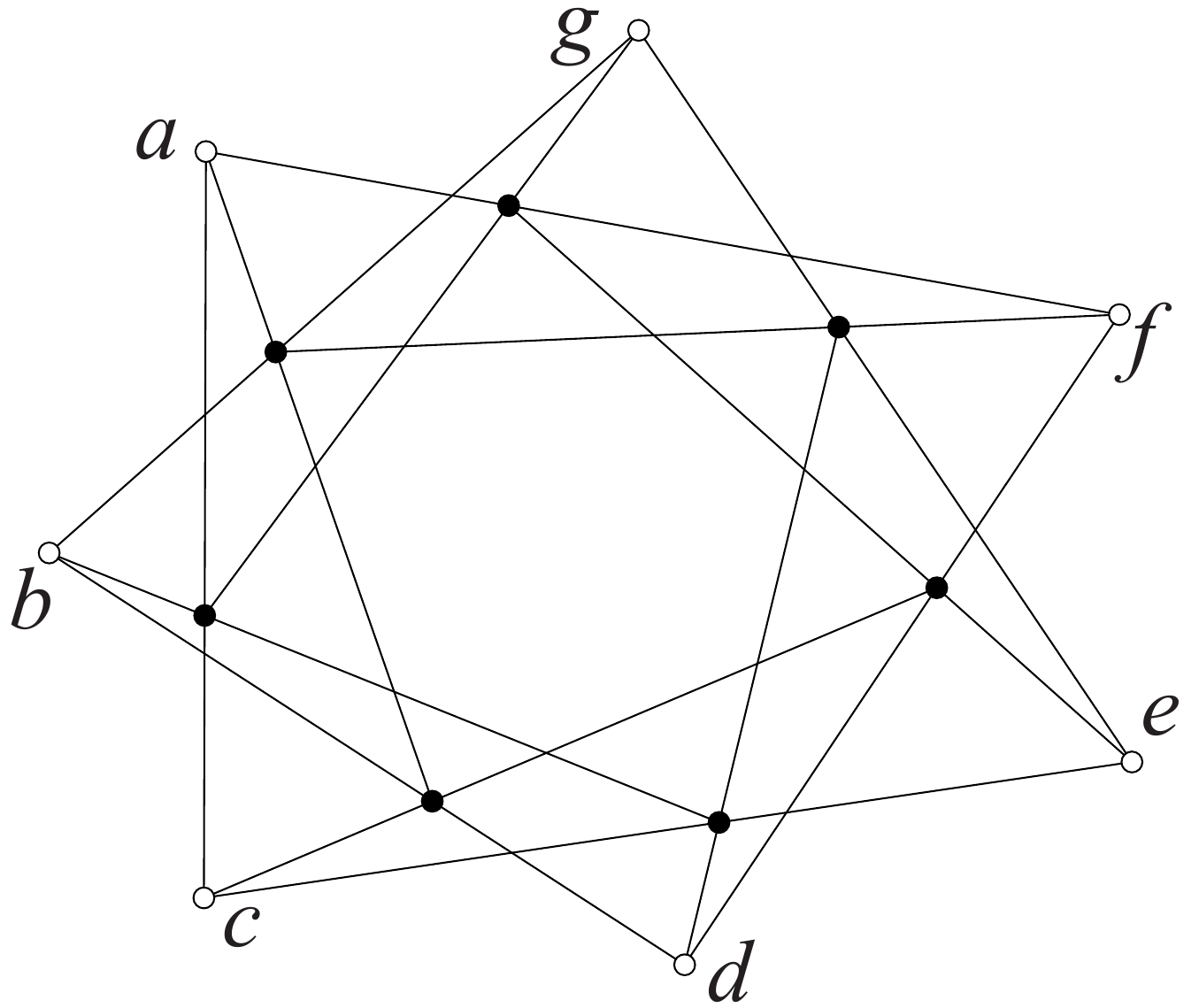}\quad
b)
\includegraphics[width=.4\textwidth]{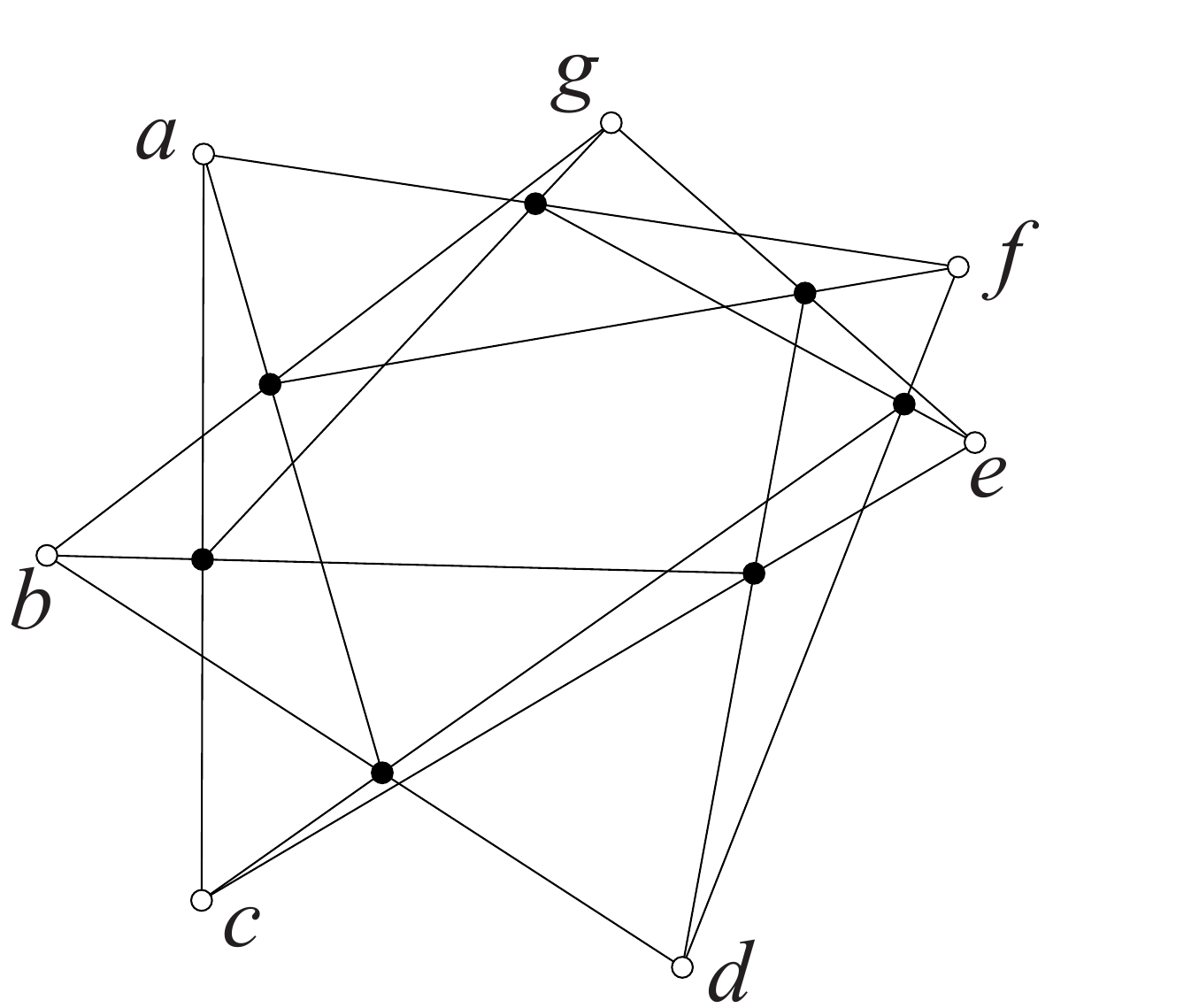}\caption{A movable configuration on 14 vertices and 14 lines constrained by 42 incidences \label{mrfourteenfig}}
\end{figure}
On the other hand, the complete quadrilateral with points $\mathbf{p}_0$, $\mathbf{p}_1$, $\mathbf{p}_2$ and $\mathbf{p}_3$
in general position  is rigid since we may pin those vertices, and take as a certificate of rigidity the Cayley algebra
expressions for the other two vertices:
$\mathbf{p}_5 = [\mathbf{p}_0\mathbf{p}_2][\mathbf{p}_1\mathbf{p}_3]$
and
$\mathbf{p}_6 = [\mathbf{p}_0\mathbf{p}_1][\mathbf{p}_2\mathbf{p}_3]$.

No Cayley algebra expression for the black vertices of the configuration of Figure~\ref{mrfourteenfig}, in terms of the seven
white vertices, is known.

\section{Relation between the different notions of rigidity}

\subsection{Infinitesimal rigidity implies rigidity}

  This exposition follows the method of Asimov and Roth~\cite{asimow1979rigidity}.
  Suppose we have points and lines of a projective configuration, with point $i$ and line $j$ incident
  if they satisfy the
  constraint
  $\mathbf{p}_i\cdot \mathbf{l}_j = -1$, and for simplicity move the configuration so that there are no points at
  infinity or lines through the origin. Suppose four of the points are pinned in general position,
  and we have a motion of points $\mathbf{p}_i(t)$ and
  lines $\mathbf{l}_j(t)$ with
  $\mathbf{p}_i(t)\cdot \mathbf{l}-j(t) = -1$ for all $t$ if $(i,j) \in I$ and $l$, and
    $\mathbf{p}_i(t) = \mathbf{p}_i(0)$ for all $t$ if point $i$ is pinned.

  If the configuration is rigid, then the only motion will be constant not only at the pins but at all
  points.
  In general, if there is a motion,  then the $n$'th derivative of the constraint gives
   $\sum_{k=0}^{n} {n \choose k} \mathbf{p}^{(k)}_i(t)\cdot \mathbf{l}^{(n-k)}_j(t) = 0$.
Evaluating at $0$ gives
   $\sum_{k=0}^{n} {n \choose k} \mathbf{p}^{(k)}_i(0)\cdot \mathbf{l}^{(n-k)}_j(0) = 0$,
   and for $k = 1$, this says
   \begin{equation} [\mathbf{p}^{(1)}_1(0), \ldots, \mathbf{l}^{(1)}_1(0), \ldots] \quad \forall \ (i,j) \in I
   \label{doweneediteq}
   \end{equation}
   is in the kernel of the rigidity matrix.
   For pinned frameworks, the condition for infinitesimal rigidity is that the only element of the kernel is the
   zero vector.

\begin{theorem}
   Suppose the pinned projective configuration is infinitesimally rigid, then for all
    $n> 0$,
   $\mathbf{p}_i^{(n)}(0) = \mathbf{0}$ for all points
   and
   $\mathbf{l}_j^{(n)}(0) = \mathbf{0}$ for all lines.
\end{theorem}

\begin{proof}
   The proof is by induction on $n$.  Infinitesimal rigidity gives the base case, so let
   $n>1$ and suppose
   $\mathbf{p}_i^{(m)}(0) = \mathbf{0}$ for all points
   and
   $\mathbf{l}_j^{(m)}(0) = \mathbf{0}$ for all lines for all $m < n$.
   The sum
   $\sum_{k=0}^{n} {n \choose k} \mathbf{p}^{(k)}_i(0)\cdot \mathbf{l}^{(n-k)}_j(0) = 0$
   has only two non-zero terms by the inductive hypothesis, reducing to the system
   $ \mathbf{p}_i(0)\cdot \mathbf{l}^{(n)}_j(0) + \mathbf{p}^{(n)}_i(0)\cdot \mathbf{l}^{(n)}_j(0)  = 0$ for
   for all incident pairs.  So the $n$'th derivatives are a solution to the pinned infinitesimal rigidity system
   for the framework, and so the $n$'th derivatives must all be zero.
\end{proof}

This shows that the system can have only constant analytic motions, hence only constant motions. See Chapter~3 of Milnor,~\cite{alma996612653408496}.

\begin{theorem}
  Every infinitesimally rigid projective framework is rigid.
\end{theorem}

So if the {\em pinned rigidity matrix}, where eight rows are added for the pins, has zero kernel,
the projective configuration must be rigid.  Equivalently, if the ordinary rigidity matrix
has a kernel of dimension exactly $8$, then the  configuration is rigid.

\subsection{Rigidity does not imply infinitesimal rigidity}

Infinitesimal rigidity is a convenient linear certificate of rigidity, however, we will show in this section that it is a strictly stronger requirement. To construct a rigid but not infinitesimally rigid configuration, we start with some inductive constructions. If $\mathbf{p}$ and $\mathbf{q}$ are the coordinates of two points (lines), then $[\mathbf{p}, \mathbf{q}]$ denotes the line containing the two points (the point of intersection of the two lines).

 \begin{theorem}\label{0exentiontheorem}
       Let $S=(P,L,I)$ be an independent configuration and suppose that $p_i$ and $p_{i'}$ are distinct points 
       with $[\mathbf{p}_i, \mathbf{p}_{i'}] \neq \mathbf{l}_j$ for any 
       line $l_j$.  Then adding a new line $l_j$ realized by $[\mathbf{p}_i, \mathbf{p}_{i'}] = \mathbf{l}_j$ and adding 
       incidences $(p_i,l_j)$ and $(p_{i'},l_j)$ yields an independent configuration.
       
       Dually, suppose that $l_j$ and $l_{j'}$ are distinct lines
       with $[\mathbf{l}_j, \mathbf{l}_{j'}] \neq \mathbf{p}_i$ for any
       point $p_i$.  Then adding a new point $p_i$ realized by $[\mathbf{l}_j, \mathbf{l}_{j'}] = \mathbf{p}_i$ and adding
       incidences $(p_i,l_j)$ and $(p_i,l_{j'})$ yields an independent configuration.
       
       If the original configuration was infinitesimally rigid, then so is the new one. 
  \end{theorem}     
     
     \begin{proof}
        We need to show independence.  By the homogeneity of the projective plane, we may assume that the points are all finite, and that the line through any pair of points does not pass through the origin, so that the matrix is defined. We may also assume that $\mathbf{p}_i$ and $\mathbf{p}_{i'}$ are linearly independent. The only
        non-zero entries of the two columns associated to the line $l_j$ are in the rows for the incidences $(p_i,l_j)$ and $(p_{i'},l_j)$, and since we assumed that $\mathbf{p}_i$ and $\mathbf{p}_{i'}$ are linearly independent, any row dependence must be zero on those rows. By assumption there is no row dependence of the remaining rows, so the configuration must be independent.  
        
        The second claim follows by duality, and infinitesimal rigidity follows from the rank of the new matrix, which is $|I| + 2$, and $|I| + 2 = 2|P| + 2(|L| + 1) - 8$.
     \end{proof}    

A \textit{dyadic rational} is a rational number whose denominator is a power of $2$. In the next theorem, we construct isostatic grids with points whose coordinates are dyadic rationals.

\begin{theorem}
   For each $N \geq 0$ there exists an independent, infinitesimally rigid
   configuration $S=(P, L, I)$
   of points and lines such
   that every geometric point/line incidence between $\mathbf{p}_i$ and $\mathbf{l}_j$ in the plane is recorded as an incidence $(i,j) \in I$, and such
   that all points with coordinates $(n/2^N: m/2^N: 1)$, for integers $n$, $m$ satisfying  $0 \leq n, m \leq 2^N$,
   is either a point of
   the configuration or is the point of intersection of at least two lines.
\end{theorem}

\begin{proof}
   We construct an example for each $N$ recursively.

   At the $0$-th level, first take the four points of the unit square
   $a = (0:0:1)$,
$b = (0:1:1)$,
$c = (1:1:1)$, and
$d = (1:0:1)$.
Then form the complete quadrilateral by constructing the lines $ab$, $ac$, $bc$ and $cd$ and the intersection $p_1=(1:0:0)$ of the lines $ad$ and $bc$, and the intersection $p_2=(0:1:0)$ of the lines $ab$ and $cd$. Note that the complete quadrangle is independent, and infinitesimally rigid.

Then add the diagonal lines, $ac$ and $bd$, and their
intersections $p_3=(1:1:0)$ and $p_4=(1:-1:0)$ with the line through the points $p_1$ and $p_2$, which is the line at infinity.
Then construct the lines $ap_4$, $bp_3$, $cp_4$ and $dp_3$. Then construct the points $e$ and $f$ as the intersections of the lines $ap_4$ and $bp_3$ and the lines $bp_3$ and $cp_4$ respectively. 

Finally, construct the horizontal line through $e$, which is the line $ep_1$, and the vertical line through $f$, which is the line $fp_2$.
See Figure~\ref{gridisostactic05wx}a.
\begin{figure}[htb]
\centering
a)
\includegraphics[scale=.375]{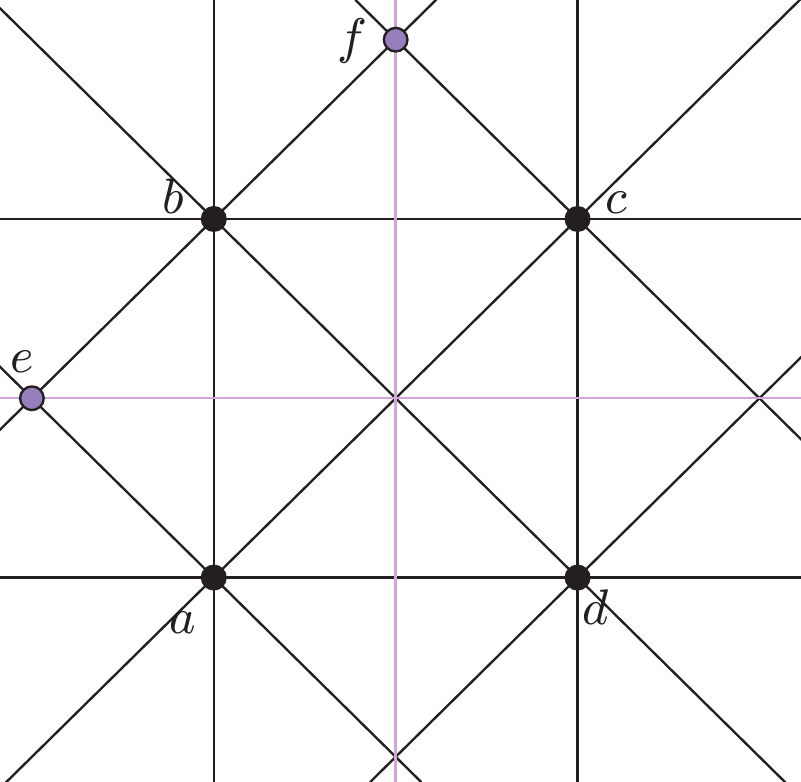}
\quad
b)
\includegraphics[scale=.375]{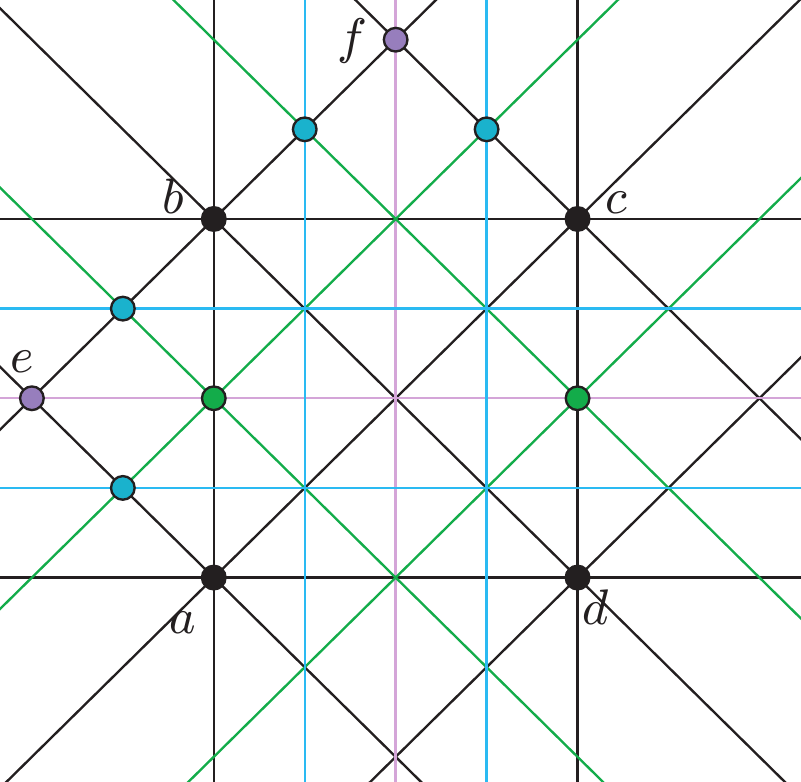}
\caption{Isostatic dyadic grids for $N=0$ and $N=1$. \label{gridisostactic05wx}}
\end{figure}
This completes the $0$-th level, and yields an independent and infinitesimally rigid configuration by Theorem~\ref{0exentiontheorem}.

Notice for the purpose of the recursion that the segments $ab$ and $cd$ have points with dyadic coordinates at level $N$ and crossings, by single lines, exactly at heights $(2k+1)/2^{N+1}$, with
$0 \leq k \leq 2^{N}-1$.
Notice also that segments $ae$, $eb$, $bf$ and $fc$ have points exactly those with dyadic coordinates at level $N+1$, and are only crossed by lines at those points.

For the recursive step start by the points of intersection of $ab$ and $cd$ with the horizontal lines of the configuration. Add the lines between each of the new points and $p_3$ and $p_4$.
Each such new line bisects one of the segments along
$ae$, $eb$, $bf$ and $fc$.  Add points at each of these intersections.  Finally, add horizontal lines
through the new points along
$ae$ and $eb$ and vertical lines through the new points along $bf$ and $fc$.  These subdivide the grid,
with the points along $ab$ and $cd$ again alternating between points and pairs of crossed lines.

The dyadic grids at levels $N=0$ and $N=1$ are illustrated in Figures~\ref{gridisostactic05wx}b and~\ref{gridisostactic05wx}a.
Each level is created from the previous by intersections of two lines and drawing lines through two points, and so each level is independent and infinitesimally rigid by Theorem~\ref{0exentiontheorem}. 
\end{proof}

In the construction, the points at level $N$, i.e. the points with coordinates $(n/2^N: m/2^N: 1)$,  in the interior of $abcd$ occur at the crossings of four lines.
If a point is introduced there by the Cayley algebra, then that point will have two incidences recorded in the set $I$, and two implied geometric incidences.  Adding either of these geometric incidences to $I$ results in a loss of constraint independence, in other words,
the introduction of an equilibrium stress. This illustrates the connection between self-stresses, constraint independence and projective theorems: certain incidences are implied by projective theorems, and such implied incidences will be dependent of the remaining incidences. 

The points at the level $N+1$ inside $abcd$ which occur at the crossings of two lines are also the points of a
square lattice, and all of these points may be added to form a rigid and independent configuration.  
See Figure~\ref{gridisostactic05wx}b.

\begin{corollary}
There exists an independent, infinitesimally rigid
   configuration $S=(P, L, I)$ in which every geometric point/line
   incidence between $\mathbf{p}_i$ and $\mathbf{l}_j$ corresponds to an incidence in $(i,j) \in I$, and containing as many points of a square lattice as required.
   \label{grid}
\end{corollary}

\begin{figure}[htb]
\centering
a)
\includegraphics[scale=.375]{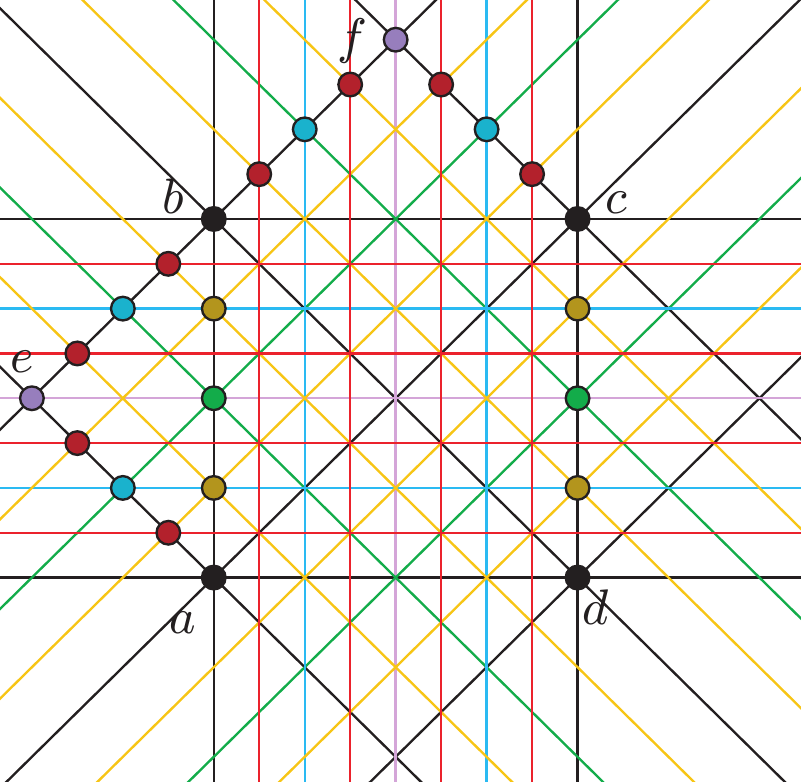}
\quad
b)
\includegraphics[scale=.375]{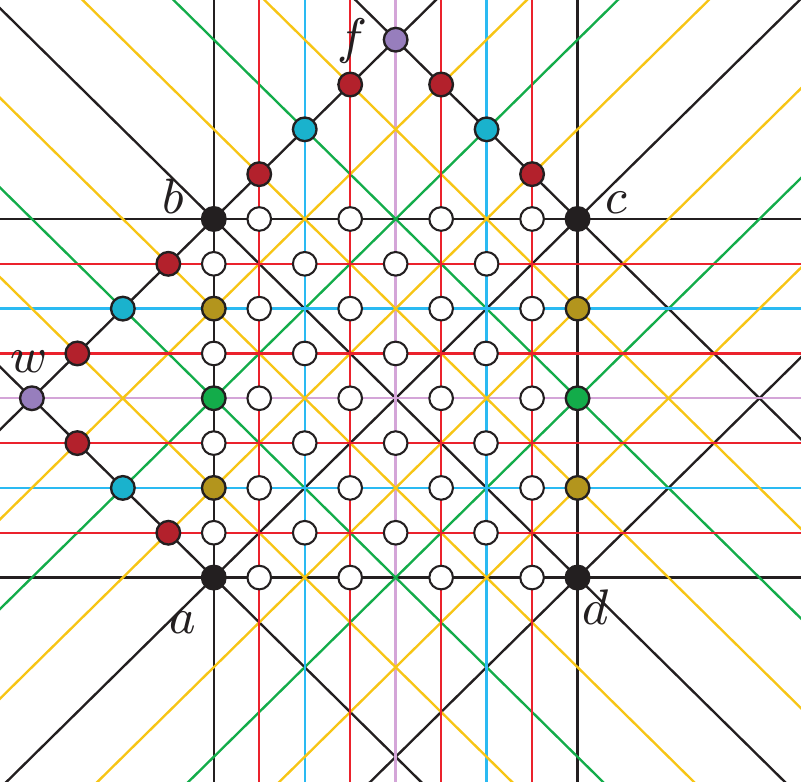}
\caption{An isostatic configuration with $40$ lattice points.\label{gridisostactic05zzz}}
\end{figure}

The isostatic configuration constructed in Corollary \ref{grid} can be further used to construct the example in the following theorem, which shows that infinitesimal rigidity does not imply rigidity; one cannot in general establish the non-rigidity of a configuration by
  considering the constraint matrix alone. 

\begin{theorem}
   There exists a configuration which is rigid but not infinitesimally rigid.
   \label{inf_not_rigid}
\end{theorem}

\begin{proof}
  This proof uses the classical method of constructing points on a quadratic curve given five points.

  Let there be a projective
   configuration as in the following figure with the white vertices pinned, from which the  black points and lines can be constructed and added to the  rigid configuration. The white points can be constructed, and pinned, as part of an isostatic grid, as in Corollary \ref{grid}. This can be done starting from four arbitrary points. See Figure~\ref{global01fig} with eight white grid pints derived from the original four. 
   \begin{figure}[htb]
   \centering
   \captionsetup{width=0.9\textwidth}
   \includegraphics[width=.6\textwidth]{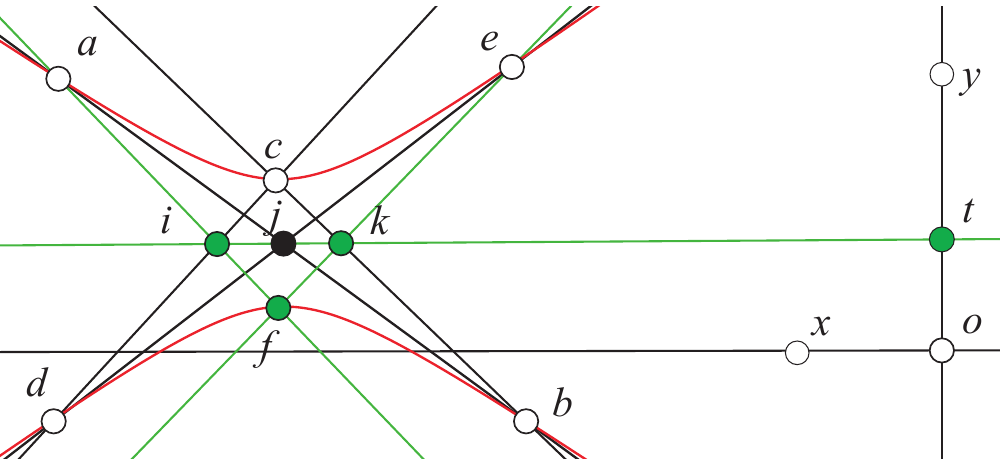} 
   \caption{A projective mechanism, pinned at the white points, whose motion requires point $f$ to move along the red hyperbola.\label{global01fig}}
   \end{figure}

   Now, add a new point $t$, in green, incident to $oy$, creating a one-degree of freedom mechanism.
The point $t$ moves freely on the fixed line $oy$, and the configuration will remain a mechanism if we add line $jt$, which we will think of as the Pascal line of the hexagon $abcdef$ which appears after constructing the additional green points and lines;
the intersection $i$ of the lines $cd$ and $jt$, the intersection $k$ of the lines $bc$ and $jt$, the lines $ai$ and $ek$, and finally the intersection $f$ of the lines $ai$ and $ek$.

Each step preserves the mechanism, and forces the point $f$ to move on the hyperbola determined by the pinned
points $\{a,b,c,d,e\}$. Finally, require an incidence between $f$ and the line $ox$.
    Consider what happens if we engineer this configuration so that the line $ox$
  cuts the hyperbola in exactly one point, a point of tangency to the hyperbola through the points
   $\{a,b,c,d,e\}$. 
   Care must be taken to get the original pinning so that the tangency can be achieved.
   
   Now the resulting configuration is rigid, and there is a unique realization which satisfies the
   final required incidence. We will show that this configuration
   is not infinitesimally rigid.  A velocity at $t$
   along $oy$ preserves the mechanism, and has consequential velocities which are zero on all but the
   yellow points and lines. These velocities must be in the kernel of the incidence rigidity matrix, so, we know that the consequential velocity at $f$ must lie along
   the pinned line $ox$. See Figure~\ref{rigidnotinfinitesimalfig}.
   \begin{figure}[htb]
   \centering
      \includegraphics[width=.7\textwidth]{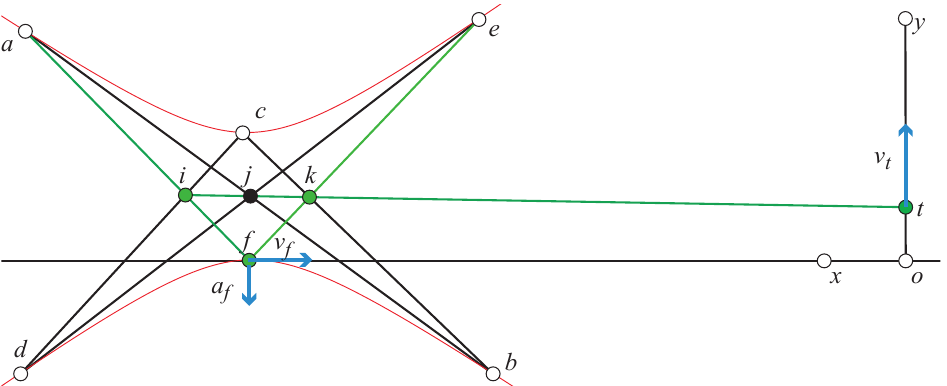}
    \caption{A rigid pinned configuration with an infinitesimal motion.\label{rigidnotinfinitesimalfig}}
    \end{figure}
   in which the velocities at $i$ and $k$, easily seen to be non-zero if $\mathbf{v}_t$ is non-zero, are not shown. All green points move infinitesimally or none.
   We know that there is
   a finite motion without the final incidence which is consistent with the constraint of the last incidence,  so
   this velocity assignment is a non-trivial infinitesimal motion of the configuration.
\end{proof}

\subsection{Second order rigidity implies rigidity}
Here we follow the method of Connelly, see~\cite{connelly1980rigidity,10.1007/BF02574003}.

In the following, we will again pin four points in general position to eliminate trivial motions.
  A second order flex is a solution to the system of equations comprising the 
  first two derivatives of the pinned configuration: 
\begin{eqnarray*}
   0 &=& \mathbf{p}_i\cdot \mathbf{l}^{(1)}_j + \mathbf{p}_i^{(1)}\cdot \mathbf{l}_j , \quad \forall (i,j) \in I
 \\ 0 &=& \mathbf{p}_i\cdot \mathbf{l}^{(2)}_j + 2\mathbf{p}^{(1)}_i\cdot \mathbf{l}^{(1)}_j + \mathbf{p}_i^{(2)}\cdot \mathbf{l}_j, \quad \forall (i,j) \in I
\end{eqnarray*}
with the derivatives and second derivatives regarded as the unknowns to be solved for.
Every solution is a pinned second order flex, and a configuration is said to be
{\em second-order rigid} if every pinned second order flex has only
zero first derivative terms.

Note that, just as infinitesimal rigidity is only a condition on the kernel of a matrix,
the second-order condition is only on the system of equations. We are essentially looking for
candidates for velocities and accelerations.

\begin{lemma}
    If the pinned configuration is second order rigid, then every motion satisfies, for all
    $n> 0$,
   $\mathbf{p}_i^{(n)}(0) = \mathbf{0}$ for all points
   and
   $\mathbf{l}_j^{(n)}(0) = \mathbf{0}$ for all lines.
\end{lemma}

\begin{proof}
   The proof is again by induction on $n$.  Second-order rigidity gives the base case, so let
   $n>1$ and suppose
   $\mathbf{p}_i^{(m)}(0) = \mathbf{0}$ for all points
   and
   $\mathbf{l}_j^{(m)}(0) = \mathbf{0}$ for all lines for all $m < n$.
   Consider now the $n$'th derivative and the $2n$'th derivative.
   The $n$'th derivative
   has only two non-zero terms by the inductive hypothesis, reducing as before to
   $ \mathbf{p}_i(0)\cdot \mathbf{l}^{(n)}_j(0) + \mathbf{p}^{(n)}_i(0)\cdot \mathbf{l}_j(0)  = 0$ for
   for all incident pairs. For the $2n$'th derivative, most terms have either a small order derivative
   of the point or the line term, and so are zero. Cancelling such terms leaves
      $\mathbf{p}_i(0)\cdot \mathbf{l}^{(2n)}_j(0)  +
   {2n \choose n} \mathbf{p}^{(n)}_i(0)\cdot \mathbf{l}^{(n)}_j(0)
   + \mathbf{p}^{(2n)}_i(0)\cdot \mathbf{l}_j(0)=0$. It is now easy to verify that the $n$'th derivative and
   the $2/{2n \choose n}$ times the $2n$'th derivative comprise a pinned second order flex, hence
   $\mathbf{p}_i^{(n)}(0) = \mathbf{l}_j^{(n)}(0) = \mathbf{0}$ for all points and lines, as required.
\end{proof}

Again, the existence of analytic motions gives
\begin{theorem}
   Every second-order rigid projective configuration is rigid.
\end{theorem}

\begin{theorem}\label{SecondOrderTheorem}
   There exists a configuration which is second order rigid but not infinitesimally rigid.
\end{theorem}

\begin{proof}
   The non-trivial velocity assignments of the configuration constructed in Theorem \ref{inf_not_rigid} are consistent with all incidences and pins. They are also consistent with the mechanism obtained if the incidence between $f$ and $ox$ is released, in which case the
   non-zero velocity $v_f$ of $f$ along $ox$ will force a non-zero centripetal acceleration at $f$, $a_t$, with $|a_f| = |v_f|^2/r$, where $r$ is the radius of curvature of the hyperbola at the point of tangency, and directed  perpendicular to $ox$, the tangent to the hyperbola.

   With the incidence from $f$ to $ox$ restored to complete the configuration, the component of acceleration at $F$ perpendicular to the line must be zero, forcing the velocity $v_t$ at $t$ which induced the only infinitesimal motion also to be zero. This says that a second order flex of the pinned configuration must be infinitesimally fixed, hence the configuration is second order rigid. 
\end{proof}

Theorem~\ref{SecondOrderTheorem} gives an algebraic proof that the configuration of Figure~\ref{rigidnotinfinitesimalfig} is rigid. Geometrically, any extension of an infinitesimal motion to a finite motion would require the point to be torn off the line by the curvature of the hyperbola.


\begin{theorem}
   There exists a rigid configuration which has a non-trivial second-order flex.
\end{theorem}

\begin{proof}
   In the proof of Theorem~\ref{inf_not_rigid}, Figure~\ref{rigidnotinfinitesimalfig} exhibits a pinned rigid projective configuration which
   is second-order rigid, and hence rigid.  That example does have a non-trivial infinitesimal flex $p$. We may use $p$ to construct a trivial
   second-order flex, $q$,
   by taking all the velocities of $q$ to be zero and all the accelerations of $q$ to be the corresponding velocities of $p$.

   So there exist rigid pinned-configurations with trivial but non-zero second-order flexes.  Now, instead starting the construction of 
   Figure~\ref{rigidnotinfinitesimalfig} with the point $x$ to be pinned point, we construct $x$ from existing pins to be rigid but having a 
   non-zero second order flex. See Figure~Figure~\ref{rigidnotsecondorder01fig}
   \begin{figure}[htb]
   \centering
   \includegraphics[width=.75\textwidth]{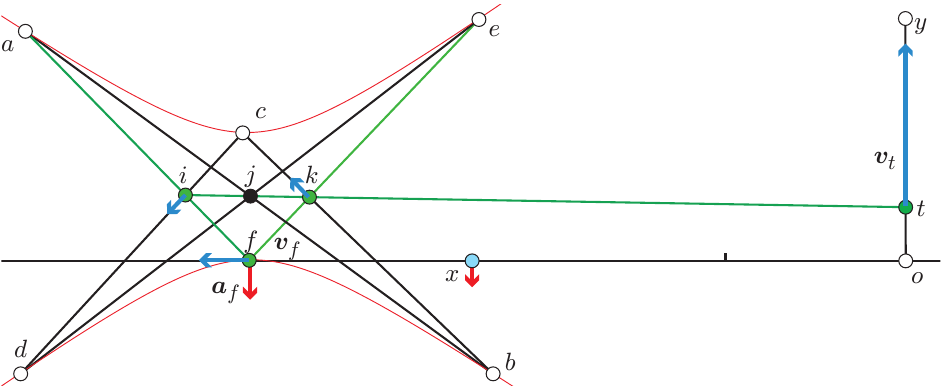}
   \caption{Resolving the acceleration by a weaker pin\label{rigidnotsecondorder01fig}}
   \end{figure}
   in which the points and lines constraining $x$ are not depicted.  
   The point $x$ still plays the role of a pin in the sense of inhibiting any finite
   motion, but $x$ may be assigned a non-zero acceleration of any magnitude which is perpendicular to the line $ox$. With this weakened pin $x$, the construction proceeds as before.  Now, however, the acceleration at $f$, which was previously inhibited by the pinned line $ox$, can be resolved, so the configuration has a non-trivial second order flex. The configuration is rigid, but second order flexible.
\end{proof}

\subsection{Alternate Realizations}\label{sec:global}

Given a rigid configuration, there are no other realizations of it in some neighborhood of the realization space.  It is natural to
ask if there are realizations farther away.  Suppose we consider realizations which, as in the previous example, begin with four points pinned in general position.

\begin{theorem}\label{globaltheorem}

   There exists a configuration whose pinned realization space consists of exactly two realizations, both of which are infinitesimally rigid, and not projectively equivalent to one another. 
\end{theorem}

\begin{proof}

Consider the configuration constructed in the proof of Theorem \ref{inf_not_rigid}, but suppose that the configuration is constructed so that $ox$ cuts the hyperbola in exactly two points. Then the configuration can be constructed in two different ways. These two configurations will not be projectively equivalent, since they can be constructed starting from the same four points.
   See the configurations of Figure~\ref{rigidnotglobal29fig}.
   \begin{figure}[htb]
   \centering
   \captionsetup{width=0.9\textwidth}\includegraphics[width=0.9\textwidth]{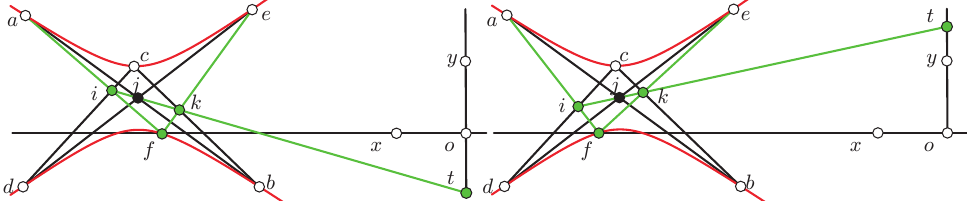} 
   \caption{There are two isolated points in the realization space of the pinned configuration.\label{rigidnotglobal29fig}}
   \end{figure}

   \end{proof}

   In the context of bar and joint rigidity, a result like Theorem~\ref{globaltheorem} would be
be expressed in terms of non-global rigidity, or the lack of unique realizability. There the connectivity
of the realization space is an important consideration. For  projective configurations, however,
any two elements of the realization space are joined by a path.
Suppose you are given two realizations, and assume without loss of generality that they have no points on
the line at infinity. The first may be dilated, shrinking it to one in which all the vertices have
the same coordinates, so keeping the directions of all lines  preserved.  In this degenerate position, the
lines are now free to rotate independently about the superimposed points and adopt the directions of the second realization.
Then, from that degenerate position, the points may be re-separated and expanded to match the second realization.
Notice that, in this path, the only non-trivial motions are in the collapsed positions. So, if we release the
four points in general position, not only is the unpinned realization space connected, it will likely contain
partially degenerate examples, and those examples may be movable.

Note that, if the line $ox$ does not meet the
hyperbola induced by the points at
$\{a,b,c,d,e\}$, then there is no realization
of the configuration at all. 
So the pinned realization space may be empty, whereas the unpinned realization space is never empty.  So pinning, useful for rigidity and infinitesimal rigidity, 
must be exercised with care if one wants to study the entire realization space.



\section{Self-stresses}

\subsection{Row dependencies: projective self-stresses}

In variations of rigidity theory, the theory of infinitesimal rigidity (column dependencies) is a dual theory of row dependencies which engineers identify as `statics'. Statics generally gives a different set of tools and insights.   What we are looking at is a set of scalars, or \emph{stress-coefficients}, $\omega_{i,j}$ assigned to the incidence of point $p_i$ with line $l_j$, and for technical  reasons we define the stress-coefficient to be zero for all non-incident pairs. 
If the vector of stress-coefficients lies in the co-kernel of the projective rigidity  matrix $M(S,\mathbf{p},\mathbf{l})$, that is, if we have
$$\sum_{k \in I}\omega_{k} Row_{k}= [0\ldots 0|0 \ldots 0],$$ where $k$ denotes the incidence recorded in the $k$'th row, $Row_{k}$, in $M(S,\mathbf{p},\mathbf{l})$,
then it  is called a \emph{(projective)}\emph{equilibrium stress} or \emph{self-stress}.

This sum gives a scalar  equation for each column -- or a vector equation for each point and each line:
$$
\sum_{j \in L}\omega_{i,j}\mathbf{p}_i = \mathbf{0},
\qquad
\sum_{i \in P}\omega_{i,j}\mathbf{l}_j = \mathbf{0}.
$$

The row space of  $M(S,\mathbf{p},\mathbf{l})$ is the set of vectors of the form
$$\sum\omega_{k} Row_{k},$$
for any scalars $\omega_{k}$.
The  conditions that the row space of $M(S,\mathbf{p},\mathbf{l})$ has rank
$2|P| + 2|L| - 8$, i.e.\ the configuration is infinitesimally rigid,
is equivalent to the condition of {\em static rigidity}, which says that
the row space of $M(S,\mathbf{p},\mathbf{l})$ spans the entire orthogonal complement of the
space of trivial infinitesimal motions.

\begin{example}\rm
We may ask if we can find the coordinates for the points and lines of a self-stress of a complete quadrangle. Since every point in a complete quadrangle is only on two lines, we see that there can only be a nonzero self-stress if all four lines coincide. If $p_j$ is on lines $l_i$ and $l_k$ then $\omega_{i,j} = -\omega_{k,j}$. Such an assignment is not unique and each assignment yields four linear homogeneous equations for the coordinates of the points which, with the assigned stress coefficients have to sum to zero. Since this system has at least 2 free variables we can get a solution for every feasible stress assignment. In such a position, the complete quadrangle is not infinitesimally rigid. It also has an actual motion, namely any point can move on the line freely with the others fixed. 
\end{example}

\begin{example}\rm
Recall from Example~\ref{subsec:pascal} that there is no row dependence in the rigidity matrix for Pascal's configuration (the matrix has full rank). This might be surprising at first in light of Pascal's theorem. However,  observe that in the 7 lines, 9 points configuration there are 6 points (the points on the hexagon) which are only on two lines. (side note: adding intersection points of two distinct lines to a given incidence structure is a 0-extension). To build the configuration, we can start with a line, put 3 points, say $a,b,c$  on it, then two more lines through the chosen points, say $l_{a_1}$, $l_{a_2}$; $l_{b_1}$, $l_{b_2}$, $l_{c_1}$, $l_{c_2}$. Then add the intersection points $1=[l_{a_1} \wedge l_{b_2}], 2=[l_{b_2} \wedge l_{c_1}], 3=[l_{c_1} \wedge l_{a_2}], 4=[l_{a_2} \wedge l_{b_1}], 5=[l_{b_1} \wedge l_{c_2}], 6=[l_{c_2} \wedge l_{a_1}]$. This configuration of 9 points and 7 lines is clearly an independent incidence structure. Pascal's theorem tells us that the points $\{1,2,3,4,5, 6\}$ lie on a conic but the conic is not a constraint! In Pappus' theorem, on the other hand, the conic shows up as a constraint via two extra lines. 
\end{example}


\subsection{The three-fold balance }
\subsubsection*{A geometric interpretation of projective equilibrium stresses}

Suppose we have an equilibrium stress in a configuration $(S, \mathbf{p}, \mathbf{l})$. For each line $l_j$ we have
$ \sum_{p_i \in P} \omega_{i,j}\mathbf{p}_i = \mathbf{0}$ and for each point $p_i$ we have
$ \sum_{l_j \in L} \omega_{i,j}\mathbf{l}_i = \mathbf{0}$,
where where the only non-zero terms correspond to incidences.

For each line $l_j$, the equilibrium equation is equivalent the two equations
\begin{equation} \sum_{p_i \in P} \omega_{i,j}\mathbf{p}_i\cdot \mathbf{l}_j = 0, \qquad
\sum_{p_i \in P} \omega_{i,j}\mathbf{p}_i \times \mathbf{l}_j = \mathbf{0}, \label{stressspliteq}
\end{equation}
The first one is trivial if the line $\mathbf{l}_j$ passes through the origin.  If not, then $(x_i,y_i) \cdot (a_j,b_j)$ is
$-1$ for any finite point $\mathbf{p}_i=(x_i:y_i:1)$, and $0$ for any point $\mathbf{p}_i=(x_i:y_i:0)$ at infinity incident to $l_j$.

\begin{figure}[htb]
\centering
\captionsetup{width=0.9\textwidth}
a)\includegraphics[scale=.82]{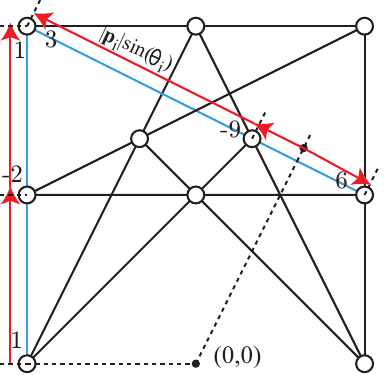}
b)\includegraphics[scale=.82]{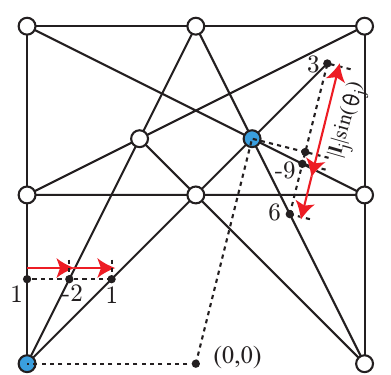}
\caption{Illustrating the three types of balance required for a projective stress. Two points and two lines are considered. The full stresses on this Desargues configuration are given in Figure~\ref{desarguesintegerfig}. a) The stresses along each line sum to zero, the combinatorial balance. In addition, the stresses, thought of as weights, balance with respect to the fulcrum indicted. (In the presence of the combinatorial balance, the fulcrum can be moved to any convenient point along the line, say the middle point.)
b) At each point, the incident line stresses much sum to zero, the combinatorial balance, and in addition the lines, again thought of as carrying weights, are balanced along a line parallel to the position vector, with the fulcrum placed nearest the point. \label{stressinterp07fig}}
\end{figure}

By considering the sum $\Sigma_{p_i \in P} \omega_{i,j}(x_i,y_i) \cdot (a_j,b_j)=0$, we get $\sum_{p_i \in P : (p_i, l_j) \in I} \omega_{i,j} = 0$ where the sum is taken over all finite points incident to $l_j$. Other than finiteness,
this balance is only by incidence and is not affected by the particular embodiment, and so we call this
equation part of the {\em combinatorial balance}, see Figure~\ref{stressinterp07fig}a.

We can get a scalar equation also from the second equation of ~(\ref{stressspliteq}) by considering the sum of norms $\sum_{p_i \in P} \omega_{i,j}|\mathbf{p}_i \times \mathbf{l}_j|=0$. This equation is equivalent to
$\sum_{i} \omega_{i,j}|\mathbf{p}_i| \sin(\theta_{i,j}) = 0$, where $\theta_{i,j}$ is the angle between the vectors $\mathbf{p}_i$ and $\mathbf{l}_j$. If we consider the vector $l_j$ as a force acting on the points incident to $l_j$, then the second equation says that the moment of the points, weighted by
    $\omega_{i,j}$, is zero with respect to the point on the line closest the origin. 
    
    If we take each vector as acting at the
    point in question, this last equation says that the moment of the points, weighted by
    $\omega_{i,j}$, is zero with respect to the point on the line closest the origin.  If one wants an interpretation in the plane, one can rewrite the equilibrium equation as
     $ \sum_{i} \omega_{i,j}\mathbf{n}_j \times(\mathbf{p}_i \times \mathbf{l}_j) = \mathbf{0}$, where $\mathbf{n}_j$ is the unit normal for $l_j$, so that the individual vectors $\mathbf{n}_j \times(\mathbf{p}_i \times \mathbf{l}_j)$ may be regarded as position vectors within the line $l_j$ emanating from the point on $l_j$ closest the origin.
 This completes the second balance. See Figure~\ref{stressinterp07fig}b. 

But the stress also sums to zero on the line coordinates, and here, not surprisingly, for each point $p_i$
 we split the equilibrium equation as
$$ \sum_{l_j \in L} \omega_{i,j}\mathbf{l}_j\cdot \mathbf{p}_i = 0, \qquad
\sum_{l_j \in L} \omega_{i,j}\mathbf{l}_j \times \mathbf{p}_i = \mathbf{0},
$$
and proceed as before. The first gives $\sum_{l_j \in L: (p_i, l_j) \in I} \omega_{i,j} = 0$, where the sum is taken over lines incident to point $p_i$
which do not go through the origin.  

Lastly, we must consider $\sum_{j} \omega_{i,j}\mathbf{l}_j \times \mathbf{p}_i = \mathbf{0}$. Each vector
can again be interpreted as providing a moment to turn the plane into the third dimension with the vector $\mathbf{p}_i$ as an axis,
see Figure~\ref{stressinterp07fig}c, and
physically we may think of the equation as expressing that the total moment here is also zero.
Again we can avoid the third dimension by $\sum_{j} \omega_{i,j}\mathbf{n}_i\times(\mathbf{l}_j \times \mathbf{p}_i) = \mathbf{0}$, where $\mathbf{n}_i$ is the unit normal
in direction $\mathbf{p}_i$
and say the total screw of the incident lines at point $r_i$ is zero.

Therefore, to be an equilibrium stress on the projective configuration, the balance of $\omega$ must be three-fold; with total moment of all incident
points on a line $l_j$ with respect to the closest point  of $l_j$ to the origin must be zero; the screws of each of the lines incident with a fixed point $p_i$ along its position line must also sum to zero, thirdly the stress must be combinatorial, with the rows of the combinatorial incidence matrix corresponding to all incidences between finite points and lines summing to zero.

\subsection*{Example: Stress on the Pappus configuration.}
In Figure~\ref{pappussintegerfig}
\begin{figure}[htb]
\centering
\captionsetup{width=0.9\textwidth}
\includegraphics[scale=.8]{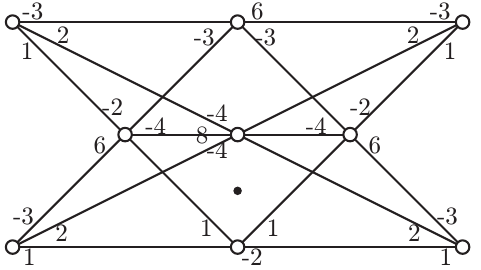}
\caption{The origin is indicated by a black dot.
The vertices are $(\pm 4, -1)$, $(\pm 2, 1)$, $(\pm 4, 3)$, $(0,-1)$, $(0,1)$ and $(0,3)$.\label{pappussintegerfig}}
\end{figure}
we have indicated the stresses at each of the $27$ point-line incidences of a Pappus configuration.

Note that the stress does not respect the vertical reflective symmetry of the figure, as
would be the case, say, in a bar-joint framework.  Compare
in particular the stress values at points $(3,0)$ and $(3,6)$, where the numerical values
assigned to the incidences are the same, but associated to incidences with unexpected lines.

To follow the calculation of the stress via the $3$-fold balance, we may begin with point $(0,-1)$,
and assign, say $1$ to the symmetrically placed diagonal line incidences, and then the combinatorial balance
requires that the horizontal line incidence be stressed by $-2$. Then the line balance on the bottom horizontal line
requires the left right point incidences to be weighted the same, both $+1$, as required by the combinatorial balance,
see Figure~\ref{pappussintegercompbfig} left.
\begin{figure}[htb]
\centering
\includegraphics[scale=.6]{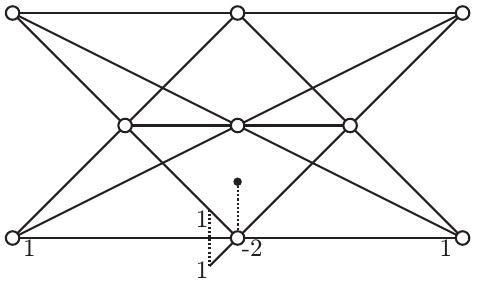}
\raisebox{.22cm}{\includegraphics[scale=.6]{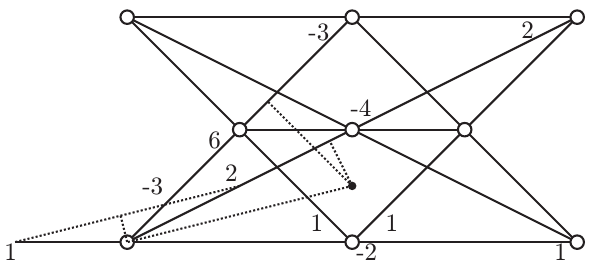}}
\caption{Initiating the balance.\label{pappussintegercompbfig}}
\end{figure}

Now we may consider the point balance at $(-4,-1)$. The fulcrum is indicated, but an easy calculation shows that,
in the presence of the combinatorial balance, the fulcrum may be taken, for the line or the point balance,
at any convenient place, say at the middle intersection point. With this choice it is apparent that the incidence
with the horizontal line stressed with $1$ balances with a stress of $2$ to the far intersection, with
the combinatorial balance giving a stress of $-3$ to the middle incidence. See Figure~\ref{pappussintegercompbfig} right.
The stress now propagates using the line balance on the two remaining lines incident to $(-4,-1)$ in an obvious way.

As one continues through the configuration, assigning stresses to incidences, no incompatibility arises, and
the stress of Figure~\ref{pappussintegerfig}, which balances all points, lines, and the combinatorics,
we find that this stress generates space of all equilibrium stresses on this configuration. \quad

\qedsymbol

\subsection*{Example: Stress on the Desargues configuration.}
In this placement of the Desargues configuration, see Figure~\ref{desarguesintegerfig}, we have chosen small integer coordinates to make the verification
via the $3$-fold balance mostly obvious.
\begin{figure}[htb]
\centering
\captionsetup{width=0.9\textwidth}
\includegraphics[scale=.75]{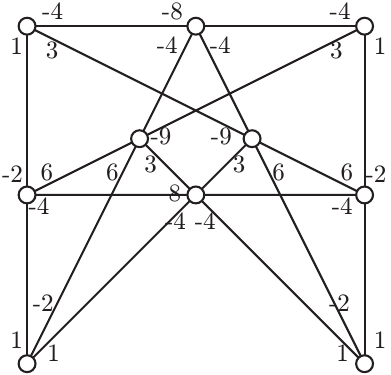}
\caption{The coordinates for the vertices are $(\pm 3, 0)$, $(\pm 3, 3)$, $(\pm 3, 6)$, $(\pm 1, 4)$, $(0,3)$ and $(0,6)$.\label{desarguesintegerfig}}
\end{figure}
To see that this is indeed the Desargues configuration, one may take the point of perspectivity to be $(-3,0)$ and the line of perspectivity to have line coordinates $(-1/3,0)$.

\section{Open problems and future work}

\subsubsection*{Self-stresses}
How can we detect a dependence of an incidence structure which represents a projective theorem not captured by counting? Self-stresses are one way.
Self-stresses imply theorems. It would be interesting to further investigate self-stress spaces of configurations. For example, if in the Desargues Configuration, the line of perspective is incident with the point of perspective, does the dimension of the self-stress space change?

Are there any new projective theorems found by self-stresses in the plane? 
What is the most flexible self-stressed configuration in the plane?  Is it characterized by the maximum undercount?


\subsubsection*{The pure conditions}

While studying  infinitesimally rigid bar and joint frameworks  explicitly as projective, the connection was captured in the work of Neal White and Walter Whiteley which showed the realizations which were generically independent became dependent on the zeros of a projectively invariant polynomial \cite{WW1,WW2}. They called this polynomial  the "pure condition". Its zeros correspond to infinitesimally flexible (or self-stressed) frameworks on graphs that satisfy the isostatic count. Does this transfer to the theory of points, lines and incidences and the associated projective rigidity matrix?

Given a configuration $S=(P,L,I)$ with $n =|I|$ with $n < 2|P|+2|L| -8$, the projective rigidity matrix might have independent rows. For the rows to be independent there must be some $ n\times n$ minor which has a nonzero polynomial as the determinant.

For the usual rigidity matrix, there is an analysis which shows that all the polynomials for different minors are, up to simple factors and signs, ``the same".  
\cite{WW1,WW2,Schulze-Whiteley-2021}.

In the previous work of White and Whiteley for bar and joint frameworks, there was a focus of ``tie-downs" which identify the columns being deleted to square-off the matrix, which can also be captured by special added rows which square up the matrix. White and Whiteley \cite{WW1,WW2} traced the impact of shifting one edge in a tie-down to a different edge.  It is a small polynomial multiplier of the initial determinant. 

In the setting of projective rigidity, we can again ask if there is an algebraic condition for projective infinitesimal flexibility of an incidence geometry that has isostatic realizations. A tie-down in the context of projective rigidity would consist of pinning four points in general position, four lines in general position, three point and a line in general position or three lines and a point in general position, see Theorem~\ref{pinningagain}.


\subsubsection*{Extensions to higher dimensions}

Our incidence structure $S=(P,L; I)$ is called an incidence geometry of rank~2. Instead of just using points and lines, we could use also higher dimensional affine subspaces of some vector space and ask about independence and dependence of incidences there. 
Point hyperplane incidence structured are well studied in algebraic geometry. Recently the existence of realizations of point-line configurations in various settings was studied in~\cite{clarke2024}.
Michel Chasles asked, about 200 years ago, for geometric conditions under which 10 given points lie on a quadric surface in 3-space~\cite{Chasles}~\cite{SW1991}, and this question is still unanswered. If enough triples of points are collinear then there are projective constructions.  A 3-dimensional analogue of Pappus's Theorem is presented in~\cite{phadke1968}. Which other theorems have higher dimensional analogues?

\subsubsection*{Projective Constructions}
Can we find nice algebraic descriptions of infinitesimal projective motions? Can we find non-trivial inductive constructions preserving motions? 
Can we describe inductive constructions preserving properties such as independence or symmetry?

\subsubsection*{Rigid $v_k$-configurations}

A combinatorial $v_k$-configuration is an incidence geometry with $v$ points and $v$ lines such that there are $k$ points on each line and $k$ lines through each point. If $k=3$ and $v > 8$, then the configuration will be flexible, as $|I|=3v < 2|P|+2|L|-8=4v-8$. Can we construct families of examples of rigid $v_k$-configurations for $k>3$?

\subsubsection*{Symmetry-induced motions}

Does there exist an example of a configuration that has a symmetry-induced motion, that is, a motion which does not exist for a realization  without that symmetry? In other words, can we find an example of an incidence geometry that has a symmetric realization, and a realization that does not have that symmetry, where the symmetric realization is more projectively flexible than the realization without the symmetry? 

\section*{Acknowledgements}
Most of our results presented here were obtained during the
Focus Program on Geometric Constraint Systems
July 1 - August 31, 2023 and we gratefully acknowledge the productive work environment provided by the Fields Institute for Research in Mathematical Sciences, as well as its financial support. The contents of this article are solely the responsibility of the authors and do not necessarily represent the official views of the Institute.

\bibliography{projective.bib}

\begin{thebibliography}{10}

\bibitem{asimow1979rigidity}
L.~Asimow and B.~Roth.
\newblock The rigidity of graphs, ii.
\newblock {\em Journal of Mathematical Analysis and Applications},
  68(1):171--190, 1979.

\bibitem{proj_paper1}
L.~Berman, S.~Lundqvist, B.~Schulze, B.~Servatius, H.~Servatius, K.~Stokes, and
  W.~Whiteley.
\newblock Projective rigidity of point-line configurations in the plane.
\newblock {\em arXiv:2407.17836}, 2024.

\bibitem{Chasles}
M.~Chasles.
\newblock {\em Aper\c{c}u historique sur l'origine et le d\'{e}veloppement des
  m\'{e}thodes en g\'{e}om\'{e}trie}.
\newblock \'{E}ditions Jacques Gabay, Sceaux, 1989.
\newblock Reprint of the 1837 original.

\bibitem{clarke2024}
Oliver Clarke, Giacomo Masiero, and Fatemeh Mohammadi.
\newblock Liftable point-line configurations: Defining equations and
  irreducibility of associated matroid and circuit varieties, 2024.

\bibitem{connelly1980rigidity}
R.~Connelly.
\newblock The rigidity of certain cabled frameworks and the second-order
  rigidity of arbitrarily triangulated convex surfaces.
\newblock {\em Advances in Mathematics}, 37(3):272--299, 1980.

\bibitem{10.1007/BF02574003}
R.~Connelly and H.~Servatius.
\newblock Higher-order rigidity--what is the proper definition?
\newblock {\em Discrete Comput. Geom.}, 11(2):193–200, December 1994.

\bibitem{Demyst}
J.~Conway and A.~Ryba.
\newblock The {P}ascal mysticum demystified.
\newblock {\em Math. Intelligencer}, 34(3):4--8, 2012.

\bibitem{CW93}
H.~Crapo and W.~Whiteley.
\newblock Autocontraintes planes et poly\`{e}dres projet\'{e}s. {I}. {L}e motif
  de base.
\newblock {\em Structural Topology}, (20):55--78, 1993.
\newblock Dual French-English text.

\bibitem{Gru2009b}
B.~Gr\"{u}nbaum.
\newblock {\em Configurations of points and lines}, volume 103 of {\em Graduate
  Studies in Mathematics}.
\newblock American Mathematical Society, Providence, RI, 2009.

\bibitem{experience}
D.~W. Henderson.
\newblock {\em Experiencing Geometry: On Plane and Sphere}.
\newblock Prentice Hall, 1995.

\bibitem{KANATANI1991333}
Kenichi Kanatani.
\newblock Computational projective geometry.
\newblock {\em CVGIP: Image Understanding}, 54(3):333--348, 1991.

\bibitem{maxwell1864xlv}
J.~C. Maxwell.
\newblock On reciprocal figures and diagrams of forces.
\newblock {\em The London, Edinburgh, and Dublin Philosophical Magazine and
  Journal of Science}, 27(182):250--261, 1864.

\bibitem{alma996612653408496}
J.~W. Milnor.
\newblock {\em Singular points of complex hypersurfaces}.
\newblock Annals of mathematics studies, no. 61. Princeton University Press,
  Princeton, N.J, 1968.

\bibitem{Schulze-Whiteley-2021}
A.~Nixon, B.~Schulze, and W.~Whiteley.
\newblock Rigidity through a projective lens.
\newblock {\em Appl. Sci}, 11(24):1--105, 2021.

\bibitem{phadke1968}
B.B. Phadke.
\newblock A three dimensional analogue of pappus's theorem.
\newblock {\em Canadian Mathematical Bulletin}, 11(5):717–717, 1968.

\bibitem{ran}
W.~J.~M. Rankine.
\newblock Principle of the equilibrium of polyhedral frames.
\newblock {\em London, Edinburgh, and Dublin Phil. Mag J. Sci}, 27(180):92,
  1864.

\bibitem{Schulze-Whiteley-2024}
B.~Schulze and W.~Whiteley.
\newblock Projective geometry of scene analysis, parallel drawing and
  reciprocal diagrams.
\newblock 2024.

\bibitem{SW1991}
B.~Sturmfels and W.~Whiteley.
\newblock On the synthetic factorization of projectively invariant polynomials.
\newblock volume~11, pages 439--453. 1991.
\newblock Invariant-theoretic algorithms in geometry (Minneapolis, MN, 1987).

\bibitem{WW2}
N.~White and W.~Whiteley.
\newblock The algebraic geometry of motions of bar-and-body frameworks.
\newblock {\em SIAM J. Algebraic Discrete Methods}, 8(1):1--32, 1987.

\bibitem{WhiteHand}
N.~L. White.
\newblock Geometric applications of the {G}rassmann-{C}ayley algebra.
\newblock In {\em Handbook of discrete and computational geometry}, CRC Press
  Ser. Discrete Math. Appl., pages 881--892. CRC, Boca Raton, FL, 1997.

\bibitem{WW1}
N.~L. White and W.~Whiteley.
\newblock The algebraic geometry of stresses in frameworks.
\newblock {\em SIAM J. Algebraic Discrete Methods}, 4(4):481--511, 1983.

\bibitem{logic1}
W.~Whiteley.
\newblock Logic and invariant theory. {I}. {I}nvariant theory of projective
  properties.
\newblock {\em Trans. Amer. Math. Soc.}, 177:121--139, 1973.

\bibitem{logic3}
W.~Whiteley.
\newblock Logic and invariant theory. {III}. {A}xiom systems and basic
  syzygies.
\newblock {\em J. London Math. Soc. (2)}, 15(1):1--15, 1977.

\bibitem{logic2}
W.~Whiteley.
\newblock Logic and invariant theory. {II}. {H}omogeneous coordinates, the
  introduction of higher quantities, and structural geometry.
\newblock {\em J. Algebra}, 50(2):380--394, 1978.

\bibitem{logic4}
W.~Whiteley.
\newblock Logic and invariant theory. {IV}. {I}nvariants and syzygies in
  combinatorial geometry.
\newblock {\em J. Combin. Theory Ser. B}, 26(2):251--267, 1979.

\bibitem{WW82}
W.~Whiteley.
\newblock Motions and stresses of projected polyhedra.
\newblock {\em Structural Topology}, (7):13--38, 1982.
\newblock With a French translation.

\bibitem{analogy}
W.~Whiteley.
\newblock An analogy in geometric homology: rigidity and cofactors on geometric
  graphs.
\newblock In {\em Mathematical essays in honor of {G}ian-{C}arlo {R}ota
  ({C}ambridge, {MA}, 1996)}, volume 161 of {\em Progr. Math.}, pages 413--437.
  Birkh\"{a}user Boston, Boston, MA, 1998.

\end{thebibliography}
\bibliographystyle{plain}

\end{document}